\documentclass[a4paper,leqno]{article}
\usepackage[latin1]{inputenc}
\usepackage{amsfonts}
\usepackage{enumitem}
\usepackage{color}
\usepackage[mathscr]{eucal}
\usepackage{textcomp}
\usepackage [english]{babel}
\usepackage{graphicx,dsfont}
\usepackage{stmaryrd}
\usepackage{amssymb,amsmath,amsthm,esint}

\newtheorem{define}{Definition}[section]
\newtheorem{pro}{Proposition}[section]
\newtheorem{Lem}{Lemma}[section]
\newtheorem{cor}{Corollary}[section]

\newtheorem{remark}{Remark}[section]
\theoremstyle{plain} \newtheorem{thm}{Theorem}[section]

\DeclareMathOperator{\dist}{dist}

\catcode`\@=11
\makeatletter

\@addtoreset{equation}{section}
\def\R{\mathbb R}

\usepackage[left=3cm,top=2cm,right=2.3cm,nohead]{geometry}
\begin{document}
{\vspace{0.001in}}

\title
{Regularity results for a class of nonlinear fractional Laplacian and singular problems}
\author{R. Arora, J. Giacomoni and G. Warnault \footnote{LMAP, UMR E2S-UPPA CNRS 5142 B\^atiment IPRA, Avenue de l'Universit\'e F-64013 Pau, France. email: rakesh.arora@univ-pau.fr, jacques.giacomoni@univ-pau.fr guillaume.warnault@univ-pau.fr}}
%
%
%
\maketitle
\begin{abstract}
In this article, we investigate the existence, uniqueness, nonexistence, and regularity of weak solutions to the nonlinear fractional elliptic problem of type $(P)$ (see below) involving singular nonlinearity and singular weights in smooth bounded domain. 
We prove the existence of weak solution in $W_{loc}^{s,p}(\Omega)$ via approximation method. Establishing a new comparison principle of independent interest, we prove the uniqueness of weak solution for $0 \leq \delta< 1+s- \frac{1}{p}$ and furthermore the nonexistence of weak solution for $\delta \geq sp.$ Moreover, by virtue of barrier arguments  we study the behavior of minimal weak solution in terms of distance function. Consequently, we prove H\"older regularity up to the boundary and optimal Sobolev regularity for minimal weak solutions.
\vspace{.2cm} 
\vspace{.1cm}
\end{abstract}
\section{Introduction}
In this paper, we study the following nonlinear fractional elliptic and singular problem
\begin{equation*}
    (P) \left\{
         \begin{alignedat}{2} 
             {} (-\Delta)^s_p u
             & {}= \frac{K_\delta(x)}{u^\gamma},\ u> 0 
             && \quad\mbox{ in } \, \Omega,           
             \\
             u & {}= 0
             && \quad\mbox{ in }\, \mathbb{R}^N \setminus \Omega   
          \end{alignedat}
     \right.
\end{equation*}
where $\Omega \subset \mathbb{R}^N$ is a smooth bounded domain with $C^{1,1}$ boundary, $s\in (0,1)$, $p\in (1,+\infty)$, $\gamma >0 $
and $K_\delta$ satisfies the growth condition: for any $x\in \Omega$ 
\begin{equation}\label{eq:sing-weight}
\frac{\mathcal{C}_1}{d^\delta(x)} \leq K_\delta(x) \leq \frac{\mathcal{C}_2}{d^\delta(x)}
\end{equation}
for some $\delta\in [0,sp)$, where, for any $x\in \Omega$, $d(x)=\dist(x, \partial \Omega)=\inf_{y\in \partial\Omega}|x-y|$. The operator $(-\Delta)^s_p$ is known as fractional $p$-Laplacian operator and defined as 
$$(-\Delta)^s_p u = 2 \lim_{\epsilon \to 0} \int_{B_\epsilon^c(x)} \frac{[u(x)-u(y)]^{p-1}}{|x-y|^{N+sp}} ~dy$$with the notation $[a-b]^{p-1}= |a-b|^{p-2}(a-b).$\vspace{0.2cm}\\
The nonlinear operator $(-\Delta)^s_p$ is the nonlocal analogue of $p$-Laplacian operator in the (weak) sense that $(1-s) (-\Delta)^s_p \to (-\Delta)_p $ as $s \to 1^-$ and for $p=2$, it reduces to the well known fractional Laplacian operator.
In particular, it is known as an infinitesimal generator of L\'evy stable diffusion process in probability and has several appearance in real life models in phase transitions, crystal dislocations, anomalous diffusion, material science, water, etc (see \cite{BucVal, KuuPala} and their reference within). 
\subsection{State of the art}
In the local setting $(s=1)$ elliptic singular problems appear in several physical models like non newtonian flows in porous media, heterogeneous catalysts. The pioneering work of Crandall et al. \cite{CRT} in the semilinear case $(s=1, p=2)$ is the starting point of the study of singular problems. Later on, much attention has been paid for the subject, leading to an abundant literature investigating a large spectrum of issues
(see for instance \cite{CRT, DMO, DiHernRako, FM, GJ, GR, HernManVega} and surveys \cite{GR1,HM}). When $s=1$, the problem $(P)$ corresponds to the following local elliptic singular problem 
\begin{equation}\label{eq:local}
(-\Delta)_p u = \frac{K(x)}{u^{\gamma}}, \ u>0 \ \text{in}\ \Omega,\ u=0\ \text{on}\ \partial \Omega.
\end{equation}

For $p=2$ and $\gamma >0$, the authors in \cite{CRT} have proved the existence and uniqueness of the classical solution in $C^2(\Omega) \cap C(\overline{\Omega})$ in case of positive bounded weight functions $K$ and have determined the boundary behavior of classical solutions.
In \cite{LazMcK}, H\"older regularity of weak solutions is established in case where $K \in C^{\alpha}(\overline{\Omega})$. In addition, they have proved that weak solutions are not $C^1$ up to boundary if $\gamma >1$ and belong to the energy space if and only if $\gamma < 3$. Furthermore, for the case of $ \gamma\in (0,1)$, Haitao in \cite{Haitao}, and Hirano et al. in \cite{HirSacShi} studied the perturbed \eqref{eq:local} with critical growth type nonlinearities in the sense of Sobolev embedding. Using Perron's method, Haitao proved multiplicity of positive weak solutions while Hirano et al. used the Nehari manifold method to achieve this goal.
Pertaining to the case when K is singular, Gomes \cite{Gomes} studied the corresponding purely singular problem and proved the existence and the uniqueness of $C^1(\overline{\Omega})$ classical solution 
via the integral representation involving the associated Green function. In \cite{DiHernRako}, Di\'az et al. considered the case where K behaves as some negative power of the distance function and proved the regularity of gradient of weak solutions in Lorentz spaces. \\ 
For the quasilinear case, i.e. $p \neq 2$, Giacomoni et al. \cite{GiaSchTak} have studied the problem \eqref{eq:local} when $K\equiv\mbox{Cst}$ and perturbed with subcritical and critical nonlinearities.  Using variational methods, the authors proved in the case $\gamma\in (0,1)$ the existence of multiple solutions in $C^{1, \alpha}(\overline{\Omega})$ and figured out the boundary behavior of weak solutions by constructing suitable sub and  supersolutions. In \cite{Giasree}, the authors proved for larger parameter $\gamma$ and subhomogeneous perturbations the existence and the uniqueness of a weak solution in $W^{1,p}_0(\Omega) \cap C(\overline{\Omega})$ if and only if  $0< \gamma < 2+ \frac{1}{p-1}$. In \cite{Can}, the authors studied the purely singular problem \eqref{eq:local} and proved the existence, the uniqueness of the very weak solution under different summability conditions on weight function $K.$ Concerning the case of $K \in L^\infty_{loc}(\Omega)$ which behaves like $\dist(x, \partial \Omega)^{-\delta}$, the authors in \cite{BougGiacHern} have proved the existence, the boundary behavior and derived the Sobolev regularity according to involved parameters. For a detailed review of elliptic equations involving singular nonlinearities we refer to the monograph \cite{GR} and the overview article   \cite{HernManVega}. \\
Later, singular problems involving the fractional Laplacian operator have been investigated in \cite{AGS} and \cite{BBMP}. Precisely, the authors studied the following singular problem 
\begin{equation}\label{eq:nonlocal}
(-\Delta)^s u = \frac{K(x)}{u^\gamma}, \ u >0 \ \text { in } \, \Omega,\ u = 0 \ \text{ in }\, \mathbb{R}^N \setminus \Omega.   
\end{equation}  
In \cite{BBMP}, Barrios et al. proved the existence and Sobolev regularity of the weak solution for the weight function $K \in L^q(\Omega)$ with $q\geq 1$. Further, in \cite{AGS}, Adimurthi et al. have studied the perturbed singular problem with  subcritical nonlinearities and singular weights of the form \eqref{eq:sing-weight} and discussed the existence, non-existence, uniqueness of classical solutions with respect to the singular parameters. To this aim, by using the integral representation via Green function and the maximum principle, they proved the sharp boundary behavior and  optimal H\"older regularity of the classical solution. Concerning the perturbed problem with critical growth nonlinearities, in \cite{TuGiSe} Giacomoni et al. tackled for any $\gamma>0$ the existence and multiplicity results in $C^\alpha_{loc}(\Omega) \cap L^\infty(\Omega)$ using non-smooth analysis theory. For further issues on nonlocal singular problems, we refer to \cite{AGGS} and references therein. 
\vspace{0.2cm}\\
The issue of regularity of weak solutions to nonlocal problems has a long history and is now quite well understood for the semilinear case $p=2$. Regularity estimates up to the boundary  apart from being relevant from itself lead to important applications as existence of solutions by Schauder fixed-point theorem and together with strong maximum principle multiplicity of solutions.  Setting
\begin{equation}\label{eq:reg}
    (-\Delta)^s_p u=f \ \text{in} \ \Omega \quad u=0 \ \text{in}\ \mathbb{R}^N \setminus \Omega,
\end{equation}
the interior regularity of the solutions was primarily resolved in case of $p=2$ by  Caffarelli and Silvestre \cite{CaffSil1, CaffSil2} using the viscosity solution approach. Boundary regularity was settled by Ros-Oton and Serra in \cite{Ros} for $f \in L^\infty(\Omega)$. For the general case $p \neq 2$, the situation is  in contrast undetermined. The local H\"older regularity is proved in \cite{Castro, Lind}. Sharpness of h\"older exponent was still open. In case where $p \geq 2$, Brasco et al. in \cite{BLS} established the optimal H\"older exponent {\it i.e.} solutions belong to $C_{loc}^{\frac{sp}{p-1}}$ when $f \in L^\infty(\Omega)$ and $\frac{sp}{p-1} <1.$ The proof of boundary regularity is a distinct issue. 
The first work regarding the nonlinear case is Iannizzotto et al. in \cite{IMS}. They proved H\"older regularity up to boundary via barrier arguments.
The result is sharp in case $p\geq 2$ whereas for $p \in (1,2)$, the optimal interior regularity in still an open question.
\vspace{0.2cm}\\
Concerning the singular problem involving $p$-fractional operator, recently Canino et al \cite{CMB} extended the work of \cite{BBMP}. They proved the existence of weak solution by approximation method and Sobolev regularity estimates.
More recently, Mukherjee et al. in \cite{TuhinaSree} studied the perturbed $p$-fractional singular problem with critical growth nonlinearities and studied existence and multiplicity of weak solutions via the minimization method under the Nehari manifold constrain.  \\
Inspired from the above works, in this paper, we study further the nonlinear fractional singular problem $(P)$ in the presence of singular weight $K_\delta$. 
To prove the existence of solutions, we first study the approximated problem $(P_\epsilon^\gamma)$ (see \eqref{eq:approx-prob}) obtained by replacing $u^{-\gamma}$ to $(u+\epsilon)^{-\gamma}$ and singular weight $K_\delta$ by $K_{\epsilon, \delta}$. Using the monotonicity properties of the approximating sequence of solutions $\{u_\epsilon\}_{\epsilon>0}$ for the problem $(P_\epsilon^\gamma)$ and by exploiting suitably the Hardy's inequality, we derive uniform apriori estimates of $u_\epsilon^{\kappa}$ in $W_0^{s,p}(\Omega)$ for some $\kappa \geq 1$ and convergence of $u_\epsilon$ to $u$, in turn to be the minimal weak solution of $(P)$ by comparison principle.  \\
The asymptotic behavior of solutions near boundary was not investigated for $p\neq 2$ in former contributions. In the nonlinear case, one can not use the integral representation with the Green function as in case $p=2$.
To overcome this difficulty, we use the arguments of constructing  explicit barrier functions. 
Due to nonlinear and nonlocal form of $(-\Delta)^s_p$, it requires to perform an involved  task of non trivial  computations. In this regard, first we define a new prototype of barrier function in $\mathbb{R}$ (and $\mathbb{R}^N$) and then explicitly compute the upper and lower estimates of $p$-fractional Laplacian acting on barrier function in $\mathbb{R}_+$ (and then $\mathbb{R}^N_+$). By exploiting the $C^{1,1}$ regularity of the boundary and local deformation arguments we exhibit the existence of sub and supersolutions in terms of distance function, giving rise to the boundary behavior of the  minimal weak solution of $(P)$. As a consequence of the boundary behavior of the $u_\epsilon$, we prove the optimal Sobolev regularity and H\"older regularity of the minimal weak solution that are new and of independent interest. The aforementioned boundary behavior of the approximated sequence $u_\epsilon$ also help us to prove the nonexistence of the weak solution for $\beta \geq sp$ that was not known in previous contributions. 
Additionally, we prove the new comparison principle for sub and  supersolution of the problem $(P)$ for $0 \leq \beta < 1+s -\frac{1}{p}.$ As an application of this comparison principle, we prove the uniqueness of the minimal weak solution in the case $0 \leq \beta < 1+s- \frac{1}{p}.$
\subsection{Function spaces and preliminaries}
Let $\Omega$ be bounded domain and for a measurable function $u: \mathbb{R}^N \to \mathbb{R}$, denote 
$$[u]_{s,p}:= \bigg(\iint_{\mathbb{R}^{2N}} \frac{|u(x)-u(y)|^p}{|x-y|^{N+sp}} ~dx ~dy\bigg)^\frac{1}{p}.$$
Define
$$W^{s,p}(\mathbb{R}^N):= \{u \in L^p(\mathbb{R}^N): [u]_{s,p} < \infty\}$$
endowed with the norm 
$$\|u\|_{s,p, \mathbb{R}^N}= \|u\|_{p} + [u]_{s,p}$$
where 
$\|.\|_{p}$ denote the $L^p$ norm. We also define
$$W_0^{s,p}(\Omega):= \{u \in W^{s,p}(\mathbb{R}^{N}): u=0 \ \text{a.e. in}\ \mathbb{R}^N \setminus \Omega \}$$
endowed with the norm 
$$\|u\|_{s,p}= [u]_{s,p}.$$ We can equivalently define $W_0^{s,p}(\Omega)$ as the closure of $C_c^\infty(\Omega)$ in the norm $[.]_{s,p}$ if $\Omega$ admits continuous boundary (see Theorem $6$, \cite{FSV}) where 
$$C_c^\infty(\Omega):= \{f: \mathbb{R}^N \to \mathbb{R}: f \in C^\infty(\mathbb{R}^N) \ \text{and}\ supp(f)  \Subset \Omega \}.$$
We also define
$$W_{loc}^{s,p}(\Omega)=\{u:\Omega \to \mathbb{R} \ |\ u \in L^p(\omega),\ [u]_{s,p,\omega} < \infty, \ \text{for all}\ \omega \Subset \Omega \}$$
where the localized Gagliardo seminorm is defined as 
$$[u]_{s,p,\omega}:= \left(\iint_{\omega \times \omega}\frac{|u(x)-u(y)|^p}{|x-y|^{N+sp}} ~dx ~dy \right)^{1/p}.$$
\begin{define}\label{def1}
A function $u \in W_{loc}^{s,p}(\Omega)$ is said to be a weak subsolution (resp. supersolution) of $(P)$, if $$u^{\kappa} \in W_{0}^{s,p}(\Omega)\  \text{for some} \ \kappa \geq 1 \ \text{and}\ \inf_K u >0 \ \text{for all} \  K \Subset \Omega$$ and
\begin{equation}\label{not}
\iint_{\mathbb{R}^{2N}} \frac{[u(x)-u(y)]^{p-1} (\phi(x)-\phi(y))}{|x-y|^{N+sp}} ~dx ~dy \leq (\text{resp.}\ \geq) \int_{\Omega} \frac{K_\delta(x)}{u^\gamma} \phi ~dx
\end{equation}
for all $\phi \in \mathbb T=\displaystyle\bigcup_{\tilde \Omega \Subset \Omega}W_0^{s,p}(\tilde \Omega)$. \\
A function which is both sub and supersolution of $(P)$ is called a weak solution to $(P)$.
\end{define}
\noindent  By virtue of the nonlinearity of the operator and the absence of integration by parts formula, such a notion of solution is considered.
Before, stating our main results, we state some preliminary results proved in \cite{BraP, CMB}:
\begin{pro}{(Lemma 3.5, \cite{CMB})}\label{prelim1}
For $\epsilon>0$ and $q>1$. Set
$$S_\epsilon^x:= \{(x,y): x \geq \epsilon, y \geq 0\}, \ S_\epsilon^y:= \{(x,y): x \geq 0, y \geq \epsilon\}.$$
Then $$|x^q-y^q| \geq \epsilon^{q-1} |x-y| \ \text{for all}\ \ (x,y) \in S_\epsilon^x \cup S_\epsilon^y.$$
\end{pro}
\begin{pro}{(Lemma 3.3, \cite{BraP})}\label{prelim2}
Let $g \in L^q(\Omega)$ with $q > \frac{N}{sp}$ and $u \in W_0^{s,p}(\Omega) \cap L^\infty(\Omega)$ satisfying
\begin{equation*}
\iint_{\mathbb{R}^{2N}} \frac{[u(x)-u(y)]^{p-1}(\phi(x)-\phi(y))}{|x-y|^{N+sp}} ~dx ~dy = \int_{\Omega} g \phi ~dx
\end{equation*}
for all $\phi \in W_0^{s,p}(\Omega).$ Then, for every $C^1$ convex  function $\Phi : \mathbb{R} \to \mathbb{R}$, the composition $w= \Phi \circ u$ satisfies
 \begin{equation*}
\iint_{\mathbb{R}^{2N}} \frac{[w(x)-w(y)]^{p-1}(\phi(x)-\phi(y))}{|x-y|^{N+sp}} ~dx ~dy \leq  \int_{\Omega} g |\Phi'(u)|^{p-2} \Phi'(u) \phi ~dx.
\end{equation*}
for all nonnegative functions $\phi \in W_0^{s,p}(\Omega).$ 
\end{pro}
\noindent Having in mind Proposition \ref{prelim1} and the condition $u^\kappa \in W_0^{s,p}(\Omega), \ \kappa \geq 1$ in definition \ref{def1}, $u$ satisfies the following definition of the boundary datum (see Proposition 1.5 in \cite{CMB}):
\begin{define}
We say that a function $u=0$ in $\mathbb{R}^N \setminus \Omega$ satisfies $u \leq 0$ on $\partial \Omega$ in sense that for $\epsilon>0$, $(u-\epsilon)^+ \in W_0^{s,p}(\Omega)$.
\end{define}

\noindent For a fixed parameter $\epsilon>0$, we define a sequence of function $K_{\epsilon, \delta}: \mathbb{R}^N \to \mathbb{R}_+$ as
\begin{equation*}
 \begin{aligned}
 K_{\epsilon, \delta}(x)=  \left\{
 \begin{array}{ll}
 (K_\delta^{-\frac{1}{\delta}}(x)+ \epsilon^{\frac{\gamma+p-1}{sp-\delta}})^{-\delta}  &  \text{ if } x \in \Omega, \\
 0 & \text{ else, } \\
 \end{array} 
 \right.
 \end{aligned}
 \end{equation*}
and $K_{\epsilon, \delta}$ is an increasing function as $\epsilon \downarrow 0$, $K_{\epsilon, \delta} \to K_\delta$ a.e. in $\Omega$ and there exist two positive constants $\mathcal{C}_3, \mathcal{C}_4$ such that, for any $x\in \Omega$,
\begin{equation}\label{qasd}
\frac{\mathcal{C}_3}{\left(d(x)+ \epsilon^{\frac{\gamma+p-1}{sp-\delta}} \right)^\delta} \leq K_{\epsilon, \delta}(x) \leq \frac{\mathcal{C}_4}{\left(d(x)+ \epsilon^{\frac{\gamma+p-1}{sp-\delta}} \right)^\delta}.
\end{equation} Define the approximated problem as
\begin{equation}\label{eq:approx-prob}
    (P_\epsilon^\gamma) \left\{
         \begin{alignedat}{2} 
             {} (-\Delta)^s_p u
             & {}= \frac{K_{\epsilon, \delta}(x)}{\left(u+ \epsilon\right)^\gamma}
             && \quad\mbox{ in } \, \Omega,           
             \\
             u & {}= 0
             && \quad\mbox{ in }\, \mathbb{R}^N \setminus \Omega.   
          \end{alignedat}
     \right.
\end{equation}
\begin{pro}\label{prelim}
For any $\epsilon >0$ and $\gamma \geq 0$, there exists a unique weak solution $u_\epsilon \in W_0^{s,p}(\Omega) \cap C^{0, \ell}(\overline{\Omega})$ of the problem $(P_\epsilon^\gamma)$ {\it i.e.} 
\begin{equation}\label{oiu}
\iint_{\mathbb{R}^{2N}} \frac{[u_\epsilon(x)-u_\epsilon(y)]^{p-1} (\phi(x)-\phi(y))}{|x-y|^{N+sp}} ~dx ~dy = \int_{\Omega} \frac{K_{\epsilon, \delta}(x)}{\left(u+ \epsilon\right)^\gamma} \phi ~dx
\end{equation}
for all $\phi \in W_0^{s,p}(\Omega)$ and for some $\ell \in (0,1).$
Moreover, the sequence $\{u_\epsilon\}_{\epsilon>0}$ satisfies  $u_\epsilon > 0$ in $\Omega$, 
$$u_{\epsilon_1}(x) < u_{\epsilon_2}(x)\ \text{for any} \ x \in \Omega \ \text{and} \ \epsilon_2 < \epsilon_1 $$ and  for any $\Omega'\Subset \Omega$, there exists $\sigma=\sigma(\Omega')>0$ such that for any $\epsilon\in (0,1)$:
\begin{equation*}
 \sigma \leq  u_1(x) < u_\epsilon(x)\ \text{for any} \ x \in \Omega'.
\end{equation*}
\end{pro}
\begin{proof}
The proof follows from Proposition 2.3, Lemma 2.4 in \cite{CMB} and Theorem 1.1 in \cite{IMS}.
\end{proof}
\subsection{Main results}
\noindent Here we describe our main results. To prove the uniqueness and nonexistence result, we establish the following comparison principle:
\begin{thm}\label{WCP}
For $0\leq  \delta < 1+s-\frac{1}{p},\ \gamma \geq 0$, let $u$ be a subsolution of $(P)$ and $\tilde{v}$ be a supersolution of $(P)$ in the sense of definition \ref{def1}. Then $u \leq \tilde{v}$ a.e. in $\Omega.$
\end{thm}
\noindent Next, we state the existence result:
\begin{thm}\label{esix-prob}
 Let $\Omega$ be a bounded domain with Lipschitz boundary $\partial \Omega$ and $ \delta\in (0, sp)$. Then,
\begin{enumerate}
\item for $\delta-s(1-\gamma)\leq 0$, then there exists a minimal weak solution $u \in W_0^{s,p}(\Omega)$ of the problem $(P)$; 
\item for $\delta-s(1-\gamma)> 0$, there exist  a minimal weak solution $u$ and a constant $\theta_0$ such that 
\begin{equation*}
u^\theta \in W_0^{s,p}(\Omega)\ \text{if}\ 
 \theta \geq \theta_0 \ \text{and} \ \theta_0 > \max\left\{1, \frac{p+\gamma-1}{p}, \Lambda \right\}
\end{equation*}
where $\Lambda:=\frac{(sp-1)(p-1+\gamma)}{p(sp-\delta)}$.
\end{enumerate} 
\end{thm}
\noindent As a consequence of comparison principle, we have the following uniqueness and nonexistence result:
\begin{cor}\label{cor:uniq}
For $0< \delta < 1+s-\frac{1}{p}$, the minimal weak solution $u$ is the unique weak solution of the problem $(P)$. 
\end{cor}
\begin{thm}\label{thm:non-exis}
Let $\delta \geq sp.$ Then there does not exist any weak solution of the problem $(P)$ in the sense of definition \ref{def1}. 
\end{thm}
\noindent Now, we state the H\"older and optimal Sobolev regularity results where we use the following exponent
$$\alpha^\star:=\frac{sp-\delta}{\gamma+p-1}.$$
\begin{thm}\label{thm:regularity}
Let $\Omega$ be a bounded domain with $C^{1,1}$ boundary and $u$ be the minimal weak solution of $(P)$. Then there exist constant $C_1, C_2>0$ and $\omega_1\in (0,s)$, $ \omega_2\in (0,\alpha^\star]$  such that
\begin{enumerate}
\item if $\delta-s(1-\gamma)\leq 0$, then $C_1 d^s \leq u \leq C_2 d^{s-\epsilon}$ in $\Omega$ and for every $\epsilon>0$ \begin{equation*} 
 \begin{aligned}
 u \in  \left\{
 \begin{array}{ll}
  C^{s-\epsilon}(\mathbb{R}^N)  & \ \text{if} \ \ 2 \leq p < \infty, \\
 C^{\omega_1}(\mathbb{R}^N)  &  \text{ if } \ 1< p < 2.\\
 \end{array} 
 \right.
 \end{aligned}
\end{equation*}
\item if $\delta-s(1-\gamma)> 0$ then $C_1  d^{\alpha^\star} \leq u \leq C_2 d^{\alpha^\star}$ in $\Omega$ and  
\begin{equation*} 
 \begin{aligned}
 u \in  \left\{
 \begin{array}{ll}
  C^{\alpha^\star}(\mathbb{R}^N)  & \ \text{if} \ \ 2 \leq p < \infty, \\
 C^{\omega_2}(\mathbb{R}^N)  &  \text{ if } \ 1< p < 2.\\
 \end{array} 
 \right.
 \end{aligned}
\end{equation*}
\end{enumerate}
\end{thm}
\begin{cor}\label{cor:bdry+sobreg}
For $\delta-s(1-\gamma)>0$ and $\Omega$ be a bounded domain with $C^{1,1}$ boundary. Then the minimal weak solution $u$ of the problem $(P)$ has the optimal Sobolev regularity: $$u \in W_0^{s,p}(\Omega)\ \text{ if and only if}\ \ \  \Lambda <1$$ and  $$ u^\theta \in  W_0^{s,p}(\Omega) \ \text{if and only if}\ \   \theta>\Lambda >1.$$
\end{cor}

\begin{remark}
In case of $\delta=0$ and $\gamma>0$, we extend the Sobolev regularity of minimal weak solution as compared to the Sobolev regularity in Theorem  $3.6$ in \cite{CMB}. Precisely, $u \in W_0^{s,p}(\Omega)$ when $\gamma \leq 1$ or $\gamma>1$ and $\Lambda <1$, and $u^\rho \in W_0^{s,p}(\Omega)$ for $\rho > \Lambda$ when $\gamma >1$ and $\Lambda \geq 1$.
\end{remark}

\subsection{Layout of the paper}
In section \ref{exis-comp}, we prove our weak comparison principle (Theorem \ref{WCP}) and existence result (Theorem \ref{esix-prob}). The main ingredients of the proof of comparison principle are by Lemma $4.1$, Theorem $4.2$ in \cite{CMB} and weak lower semi continuity of the perturbed energy functional near singularity. Using the monotonicity of the approximated solution $u_\epsilon$ of the problem $(P_\epsilon^\gamma)$ obtained from Proposition \ref{prelim}, we derive uniform estimate of $u_\epsilon^\kappa$ in $W_0^{s,p}(\Omega)$ for some $\kappa \geq 1$ and finally, by exploiting the Hardy inequality, we prove the convergence of $u_\epsilon$ to a minimal weak solution $u$.  \\
The hard and tricky work is carried out in section \ref{distance}. Using the arguments of barrier functions, we construct sub and  supersolution of the purely singular weight problem near the boundary. For this, we introduce the new energy space and notion of weak energy solution, subsolution and supersolution. We start by computing the $p$-fractional Laplacian of $V_\lambda(x):= (x_N+ \lambda^{\frac{1}{\alpha}})_+^\alpha$ in $\mathbb{R}^N$ and derive lower and upper estimates of $(-\Delta)^s_p V_\lambda$ in $\mathbb{R}^N_+$ (Theorem \ref{thest}, Corollary \ref{esthm}). Next, we prove the above upper and lower estimates are preserved under a smooth change of variables via a $C^{1,1}$ diffeomorphism close to identity in the isomorphic image of a set close to the boundary of $\mathbb{R}^N_+$ (Theorem \ref{thpr}). Then, by smoothly extending the distance function outside $\Omega$ (see \eqref{ext:def}), we construct sub and supersolution $\underline{w}_\rho$ and $\overline{w}_\rho$ for some $\rho >0$ (see \eqref{sub}, \eqref{super} satisfying \eqref{mainest}). \\
In section \ref{Sob-Hol}, we study the purely singular weight problem $(S_0^\delta)$ and we prove the existence of minimal weal solution (Theorem \ref{weithm}) and boundary behavior of the minimal weak solution (Theorem \ref{bdybeh}). As a n application of Theorem \ref{bdybeh}, we prove the boundary behavior, Optimal Sobolev and  H\"older regularity of the minimal weak solution to our main problem $(P)$ (Theorem \ref{thm:regularity} and Corollary \ref{cor:bdry+sobreg}).  \\
In section \ref{non-es}, we prove the non-existence of weak solution for $\delta >sp$, as a consequence of weak comparison principle (Theorem \ref{WCP}) and boundary behavior of weak solution (Theorem \ref{thm:regularity}). In section \ref{ann}, we recall the local regularity results for bounded forcing term from \cite{BLS, IMS}.
\section{Comparison principle and existence result}\label{exis-comp}
In this section, we prove the weak comparison principle and existence result concerning the problem $(P).$\\
\noindent \textbf{Proof of Theorem \ref{WCP}}: The proof is almost identical as the proofs of Lemma 4.1 and Theorem 4.2 in  \cite{CMB}. For the reader's convenience, we precise some details to explain the restriction on $\delta$. More precisely, we need a minimizer belonging to $\mathcal{L}:= \{\phi \in W_0^{s,p}(\Omega):\ 0 \leq \phi \leq \tilde{v} \ \text{a.e. in} \ \Omega \}$ of the following energy functional defined on $W_0^{s,p}(\Omega)$ as, for $\epsilon>0$
$$J_\epsilon(w):= \frac{1}{p} \iint_{\mathbb{R}^{2N}} \frac{|w(x)-w(y)|^p}{|x-y|^{N+sp}} ~dx ~dy -  \int_{\Omega} K_\delta(x) G_\epsilon(w) ~dx$$
where $G_\epsilon$ is the primitive such that $G_\epsilon(1)=0$ of the function $g_\epsilon$ defined by
\begin{equation*}\label{propq}
 \begin{aligned}
 g_\epsilon(t)=  \left\{
 \begin{array}{ll}
 \min\left\{\frac{1}{t^\gamma}, \frac{1}{\epsilon}\right\}  &  \text{ if } t >0, \\
 \frac{1}{\epsilon} & \text{ if } t \leq 0.\\
 \end{array} 
 \right.
 \end{aligned}
 \end{equation*}
Let $\{w_n\} \subset W_0^{s,p}(\Omega)$ be such that $w_n \rightharpoonup w$ in $W_0^{s,p}(\Omega)$. Let $\nu \in (0,1)$ small enough such that $\frac{1-\nu}{p} + \frac{\nu}{q} + \frac{1}{r}= 1$ where $q < p_s^\ast:=\frac{Np}{N-sp}$ if $N>sp$ and $(s(1-\nu) -\delta)r > -1$ (since $\delta< 1 + s- \frac{1}{p}$).\\
Hence $x\mapsto d^{s(1-\nu) -\delta}(x)\in L^r(\Omega)$ and by using H\"older and Hardy inequalities (see Theorem 1.4.4.4 and Corollary 1.4.4.10 in \cite{Gri}), we obtain
\begin{equation*}   
\begin{split}
\int_{\Omega} \frac{|w_n-w|}{d^\delta(x)} ~dx &= \int_{\Omega} \left(\frac{|w_n-w|}{d^s(x)}\right)^{1-\nu} |w_n -w|^\nu d^{s(1-\nu)-\delta}(x) ~dx \\
& \leq C \|w_n -w\|_{s,p}^{1-\nu} \|w_n-w\|_{L^q(\Omega)}^{\nu } 
\end{split}
\end{equation*}
for some constant $C>0$ independent of $w_n$ and $w$.\\
Since $W_0^{s,p}(\Omega)$ is compactly embedded in $L^q(\Omega)$ for $q< p_s^\ast$, $\|w_n -w\|_{s,p}$ is uniformly bounded in $n$ and $\|w_n -w\|_{L^q(\Omega)} \to 0$ as $n\rightarrow \infty$.\\
Finally, gathering the lower semicontinuity of $[.]_{s,p}$ and $G_\epsilon$ globally Lipschitz, we deduce that $J_\epsilon$ is weakly lower semicontinuous in $ W_0^{s,p}(\Omega)$ and admits a minimizer $w_0$ on  $\mathcal{L}$.\\  
The rest of the proof follows exactly the proofs of Lemma 4.1 and Theorem 4.2 in  \cite{CMB} and we obtain 
\begin{equation*}
u \leq w_0 \leq \tilde{v} \ \text{in} \ \Omega.
\end{equation*}
By following the same idea of proof, we can prove it for $\gamma=0.$
\qed
\\
\ \\
Now we prove our existence and uniqueness result:\vspace{0.2cm} \\
\noindent \textbf{Proof of Theorem \ref{esix-prob}}: Let $u_\epsilon \in W_0^{s,p}(\Omega)$ be the weak solution of $(P_\epsilon^\gamma).$ 
{Adapting the proofs of Theorem 3.2 and 3.6 in \cite{CMB},} 
it is sufficient to verify the sequences  $\{u_\epsilon\}_\epsilon$ in the case $\delta-s(1-\gamma)\leq0$ and $\{u^\theta_\epsilon\}$ for a suitable parameter $\theta >1$ in the case $\delta-s(1-\gamma)> 0$ are  bounded in $W^{s,p}_0(\Omega)$ and the convergence of the right-hand side in \eqref{oiu}.\\ 
\textbf{Case 1:}  $\delta-s(1-\gamma)\leq 0$.\\ 
The condition implies $\gamma<1$ hence taking $\phi= u_\epsilon$ in \eqref{oiu} and applying H\"older and Hardy inequalities {(see Theorem 1.4.4.4 and Corollary 1.4.4.10 in \cite{Gri})} , we obtain
\begin{equation}\label{e4}
[u_\epsilon]_{s,p}^p \leq  \mathcal{C}_2 \int_{\Omega} d^{s(1-\gamma)-\delta}(x) \left(\frac{u_\epsilon}{d^{s}(x)}\right)^{1-\gamma} ~dx \leq C \|\frac{u_\epsilon}{d^{s}}\|_{L^p(\Omega)}^{1-\gamma} \leq C\ [u_\epsilon]_{s,p}^{1-\gamma}
\end{equation}
which implies $\|u_\epsilon\|_{s,p} \leq C < \infty$.\\ 
\textbf{Case 2:} $\delta-s(1-\gamma)> 0$\\
Let $\Phi: \mathbb{R}_+ \to \mathbb{R}_+$ be the function defined as  
$\Phi(t)= t^{\theta}$ for some 
\begin{equation}\label{e18}
 \theta > \max\left\{1,\frac{p+\gamma-1}p, \Lambda\right\}.
\end{equation}
For any $\epsilon>0$, choosing $g= \frac{K_{\epsilon, \delta}}{(u_\epsilon + \epsilon)^\gamma} \in L^\infty(\Omega)$ and $w= \Phi \circ u_\epsilon$ in Proposition \ref{prelim2}, we obtain
\begin{equation}\label{e05}
\iint_{\mathbb{R}^{2N}} \frac{[\Phi(u_\epsilon)(x)-\Phi(u_\epsilon)(y)]^{p-1} (\phi(x)-\phi(y))}{|x-y|^{N+sp}} ~dx ~dy\leq  \int_{\Omega} \frac{K_{\epsilon, \delta}(x)}{(u_\epsilon+ \epsilon)^\gamma} |\Phi'(u_\epsilon)|^{p-2} \Phi'(u_\epsilon) \phi ~dx
\end{equation}
for all nonnegative functions $\phi \in W_0^{s,p}(\Omega)$. Since $u_\epsilon \in W_0^{s,p}(\Omega) \cap L^\infty(\Omega)$ and $\Phi$ is locally Lipschitz, therefore $\Phi(u_\epsilon) \in W_0^{s,p}(\Omega).$ Then by choosing $\phi= \Phi(u_\epsilon)$ as a test function in \eqref{e05}, we get
\begin{equation}\label{e6}
\begin{split}
\iint_{\mathbb{R}^{2N}} \frac{|\Phi(u_\epsilon)(x)-\Phi(u_\epsilon)(y)|^{p}}{|x-y|^{N+sp}} ~dx ~dy &\leq \int_{\Omega} \frac{K_{\epsilon, \delta}(x)}{(u_\epsilon+ \epsilon)^\gamma} |\Phi'(u_\epsilon)|^{p-2}  \Phi'(u_\epsilon) \Phi(u_\epsilon) ~dx\\
& \leq \mathcal{C}_2 \int_{\Omega} \frac{1}{ d^\delta(x)} \frac{|\Phi'(u_\epsilon)|^{p-2}  \Phi'(u_\epsilon) \Phi(u_\epsilon)}{u_\epsilon^\gamma} ~dx. 
\end{split}
\end{equation}
Now, for any $\epsilon>0$, there exists a constant $C$ independent of $\epsilon$ such that 
\begin{equation}\label{e7}
\frac{|\Phi'(u_\epsilon)|^{p-2}  \Phi'(u_\epsilon) \Phi(u_\epsilon)}{u_\epsilon^\gamma} \leq C (\Phi(u_\epsilon))^{\frac{\theta p - (p+\gamma-1)}{\theta}}.
\end{equation}
where $\frac{\theta p - (p+\gamma-1)}{\theta}>0$ since $\theta> \frac{p+\gamma-1}{p}$. By combining \eqref{e6}-\eqref{e7},  we obtain applying H\"older and Hardy inequalities: 
\begin{equation*}
\begin{split}
[\Phi(u_\epsilon)]_{s,p}^p & \leq C \int_{\Omega} d^{\frac{s(\theta p - (p+\gamma-1))}{\theta}-\delta}(x) \left(\frac{\Phi(u_\epsilon)}{d^s}(x)\right)^{\frac{\theta p - (p+\gamma-1)}{\theta}} ~dx \\
&\leq C \left(\int_{\Omega} d^{\frac{sp(\theta-\Lambda)-\theta}{\Lambda}}(x) ~dx \right)^{\frac{p+\gamma-1}{\theta p}} \left(\int_{\Omega} \left(\frac{\Phi(u_\epsilon)}{d^s(x)}\right)^{p} ~dx \right)^{\frac{\theta p - (p+\gamma-1)}{\theta p }}\\
& \leq C \ [\Phi(u_\epsilon)]_{s,p}^{\frac{\theta p - (p+\gamma-1)}{\theta}}
\end{split}
\end{equation*}
and we conclude  $\left\{\Phi(u_\epsilon)\right\}_{\epsilon >0}$ is bounded in $W_0^{s,p}(\Omega)$.\\
Finally, let $\tilde \Omega \Subset\Omega$, and   $\phi \in W^{s,p}_0(\tilde \Omega)$. 
By Proposition \ref{prelim}, there exists a constant $\eta_{\tilde \Omega}$ such that for any $\epsilon>0$,
\begin{equation*}\label{e0} u_\epsilon(x) \geq \eta_{\tilde \Omega}, \ \text{for a.e. in } \tilde \Omega.
\end{equation*} 
By the previous inequality, we have 
\begin{equation}\label{e15}
\left| \frac{K_{\epsilon, \delta}(x) \phi}{(u_\epsilon + \epsilon)^\gamma}\right| \leq \eta^\gamma_{\tilde \Omega}  M |\phi|
\end{equation} 
where $M= \dist^{-\delta}(\tilde \Omega, \Omega)$, hence we get by Dominated convergence theorem:
\begin{equation}\label{e16}
\int_{\Omega} \frac{K_{\epsilon, \delta}(x)}{(u_\epsilon+\epsilon)^\gamma} \phi ~dx\to \int_{\Omega} \frac{K_\delta(x)}{u^\gamma} \phi ~dx
\end{equation}
where $u:=\lim_{\epsilon \to 0} u_\epsilon$. The rest of the proof follows exactly the end of the proofs of Theorem 3.2 and 3.6 in \cite{CMB}.\\
Finally, for any $\epsilon >0$, $u_\epsilon \leq v$ {\it a.e.} in $\Omega$ where $v$ is another weak solution of $(P)$. Indeed, $v$ is a weak supersolution in sense of Definition \ref{def1} of the problem $(P_\epsilon^\gamma)$ hence Theorem 4.2 in \cite{CMB} implies the inequality. Passing to the limit $\epsilon \to 0$ gives that $u$ is a minimal solution.
\qed 
\begin{remark}\label{rk31}
The proof of Case 1 holds assuming $\Lambda\leq 1$ and $\gamma<1$. Indeed, $d^{s(1-\gamma)-\delta}\in L^{\frac{p}{p-1+\gamma}}(\Omega)$ and we obtain \eqref{e4}.
\end{remark}
\begin{remark}
In case of $\delta=0$, the Sobolev regularity of the minimal weak solution in Theorem \ref{esix-prob} coincides with the Sobolev regularity in Theorem $3.2$ for $\gamma \leq 1$ and Theorem $3.6$ in \cite{CMB} for $\gamma >1$ by taking $\theta= \frac{p+\gamma-1}{p}.$
\end{remark}

\noindent\textbf{Proof of Corollary \ref{cor:uniq}} \\
Let $u_1, u_2$ are two solution of the problem $(P).$ Then by considering $u_1$ and $u_2$ as a subsolution and supersolution respectively in Theorem \ref{WCP}, we get $u_1 \leq u_2$ in $\Omega$ for $0< \delta < 1+s-\frac{1}{p}.$ Now, by reversing the role of $u_1$ and $u_2$, we obtain $u_1=u_2.$
\qed
\section{Construction of barrier functions}\label{distance}
In this section, we construct explicit sub and  supersolution for the following problem 
\begin{equation*}
      (S_{0}^\delta) \left\{
         \begin{alignedat}{2} 
             {} (-\Delta)^s_p u(x)
             & {}= K_\delta(x)
             && \quad\mbox{ in } \, \Omega,           
             \\
             u & {}= 0
             && \quad\mbox{ in }\, \mathbb{R}^N \setminus \Omega.   
          \end{alignedat}
     \right.
\end{equation*}
Before that, we introduce the new notion of weak solution and corresponding vector space:\\ 
Let $\Omega \subset \mathbb{R}^N$ be bounded. We define
$$\overline{W}^{s,p}(\Omega):=\left\{u \in L^p_{loc}(\mathbb{R}^N):\  \exists\ K \ \text{s.t.} \ \Omega \Subset K, \|u\|_{W^{s,p}(K)} + \int_{\mathbb{R}^N} \frac{|u(x)|^{p-1}}{(1+|x|)^{N+sp}}~dx < \infty \right\}$$
where $ \|u\|_{W^{s,p}(\Omega)}=\|u\|_{L^p(\Omega)}+[u]_{s,p,\Omega}$.
If $\Omega$ is unbounded, we define
$$\overline{W}_{loc}^{s,p}(\Omega):= \{u \in L^p_{loc}(\mathbb{R}^N): u \in \overline{W}^{s,p}(\tilde{\Omega}),\ \text{for any bounded }\ \tilde{\Omega} \subset \Omega\}.$$ 
\begin{define}\label{notion}
(\textbf{Weak energy Solution}) 
Let $f \in L^{p'}(\Omega)$ where $p'$ is the conjugate exponent of $p$ and $\Omega$ be a bounded domain. We say that $u \in \overline{W}^{s,p}(\Omega)$ is a weak energy solution of $(-\Delta)^s_p\,u= f$ in $\Omega$, if
{$$(-\Delta)^s_p\, u =  f \ \ \  \text{E-weakly in}\ \Omega$$
that is}
$$ \iint_{\mathbb{R}^{2N}} \frac{[u(x)-u(y)]^{p-1}(\phi(x)-\phi(y))}{|x-y|^{N+sp}} ~dx~dy= \int_{\Omega} f(x) \phi(x)~dx  $$
for all $\phi \in W_0^{s,p}(\Omega)$ and a function $u$ is a weak energy subsolution (resp. weak energy supersolution) of  $(-\Delta)^s_p\,u= f$ in $\Omega$, if 
$$(-\Delta)^s_p\, u \leq (resp. \geq) \  f \ \ \  \text{E-weakly in}\ \Omega$$
that is
$$ \iint_{\mathbb{R}^{2N}} \frac{[u(x)-u(y)]^{p-1}(\phi(x)-\phi(y))}{|x-y|^{N+sp}}~dx~dy  \leq (resp. \geq) \int_{\Omega} f(x) \phi(x)~dx  $$
for all $\phi \in W_0^{s,p}(\Omega), \phi \geq 0.$\\
If $\Omega$ is unbounded we say that $u \in \overline{W}^{s,p}_{loc}(\Omega)$ is a weak energy solution (weak energy subsolution/weak energy supersolution) of $(-\Delta)^s_p(u) = (\leq/\geq)\ f$ in $\Omega$, if it does so in any open bounded set $\Omega' \subset \Omega.$
\end{define} 
\noindent  For any $\alpha\in(0,s)$, we define 
$$\beta:= sp- \alpha(p-1).$$
We start by computing the upper and lower estimates in the half line $\mathbb{R}_+:= \{x \in \mathbb{R}: x >0\}$ of  $(-\Delta)^s_p$ of the function $U_\lambda(x):=\left((x+\lambda^{\frac{1}{\alpha}})^+\right)^\alpha$, $\lambda\geq 0$ defined in $\R$.\\
We recall the notation, for any $t\in \R$, $ [t]^{p-1}= |t|^{p-2} t$. 
\begin{thm}\label{thest}
Let $\lambda \geq 0$, $\alpha\in(0,s)$ and $p>1$. Then, there exist two positive constants $C_1,\,C_2 >0$ depending upon $\alpha,\ p$ and $s$ such that
\begin{equation}\label{est1}
C_1 (x+ \lambda^{\frac{1}{\alpha}})^{-\beta} \leq (-\Delta)^s_p U_\lambda(x) \leq C_2 (x+ \lambda^{\frac{1}{\alpha}})^{-\beta} \  \text{pointwisely in} \ \mathbb{R}_+.
\end{equation}
Moreover,  for $\lambda>0$, $U_\lambda \in \overline{W}^{s,p}_{loc}(\mathbb{R}_+)$ and for $\lambda=0$, $U_\lambda \in \overline{W}^{s,p}_{loc}(\mathbb{R}_+)$ if $s-\frac{1}{p} < \alpha <s$.
\end{thm}
\begin{proof}
Let $x \in \mathbb{R}_+$ and let $\epsilon\in \R$ such that $|\epsilon|<x$. We have
\begin{equation*}
\begin{split}
 \int_{\R\setminus (x-|\epsilon|,x+|\epsilon|)} \frac{[U_\lambda(x)-U_\lambda(z)]^{p-1}}{|x-z|^{1+sp}} ~dz
& = \left(\int_{-\infty}^{-\lambda^{\frac{1}{\alpha}}} \dots + \int_{-\lambda^{\frac{1}{\alpha}}}^{x-|\epsilon|} \dots + \int_{x+|\epsilon|}^\infty \dots \right)\\
&= (x+ \lambda^{1/\alpha})^{-\beta} \mathcal{P_\epsilon}(x)
\end{split}
\end{equation*}
where, by the change of variable $y=\frac{z+\lambda^{1/\alpha}}{x+ \lambda^{1/\alpha}}$:
\begin{align*}
\mathcal{P_\epsilon}(x):=&(x+ \lambda^{1/\alpha})^{sp} \int_{-\infty}^{- \lambda^{1/ \alpha}} {\frac{1}{|x-z|^{1+sp}} ~dz} + \int_{0}^{1-\frac{|\epsilon|}{x + \lambda^{1/\alpha}}} \frac{[1 - y^\alpha]^{p-1}}{|1-y|^{1+sp}} ~dy\\
&+ 
\int_{1+\frac{|\epsilon|}{x + \lambda^{1/\alpha}- |\epsilon|}}^\infty \frac{[1 - y^\alpha]^{p-1}}{|1-y|^{1+sp}} ~dy+ \int_{1+\frac{|\epsilon|}{x + \lambda^{1/\alpha}}}^{1+\frac{|\epsilon|}{x + \lambda^{1/\alpha}-|\epsilon|}} \frac{[1 - y^\alpha]^{p-1}}{|1-y|^{1+sp}} ~dy\\
&:=\mathcal{P}_1(x) + \mathcal{P}_2(x, \epsilon) + \mathcal{P}_3(x, \epsilon) + \mathcal{P}_4(x, \epsilon).
\end{align*}
To conclude \eqref{est1}, it suffices to obtain a uniform estimate of $\mathcal{P_\epsilon}$ in $\mathbb{R}_+$. First we note
\begin{equation}\label{est01}
\begin{split}
\mathcal{P}_1(x)&= (x+ \lambda^{1/\alpha})^{sp} \int_{-\infty}^{- \lambda^{1/ \alpha}} {\frac{1}{|x-z|^{1+sp}} ~dz}  = \frac{1}{sp}.
\end{split}
\end{equation}
Moreover, the change of variable $y \to \frac{1}{y}$ in $\mathcal P_{3}$ yields:
$$\mathcal P_{3}(x,\epsilon)=- \int_0^{1-\frac{|\epsilon|}{x + \lambda^{1/\alpha}}} \frac{(1 - y^\alpha)^{p-1} y^{\beta -1} }{|1-y|^{1+sp}} ~dy.$$
Hence
%
\begin{equation}\label{2-3}
\mathcal P_{2,3}(x,\epsilon):=\mathcal P_{2}(x,\epsilon)+\mathcal P_{3}(x,\epsilon)= \int_0^{1-\frac{|\epsilon|}{x + \lambda^{1/\alpha}}} \frac{(1 - y^\alpha)^{p-1} (1- y^{\beta -1}) }{|1-y|^{1+sp}} ~dy. 
\end{equation}
We consider two cases to estimate $\mathcal P_{2,3}$:\\ 
\textbf{Case 1:}  $\beta<1$.\\
First, note in this case, $\mathcal P_{2,3}(x,\epsilon)\leq 0$, it suffices to estimate $\mathcal P_{2,3}$ from below. \\
 There exists  $ \tilde{s}\in (s,1)$ such that $\beta > \tilde{s}$ hence for any $y\in (0,1)$:
 $$y^{\beta-1}-1 \leq y^{\tilde{s}-1}-1,\ (1-y^{\alpha})\leq (1-y^{\tilde{s}})\ \mbox{ and }\  \frac{1}{(1-y)^{1+sp}} \leq \frac{1}{(1-y)^{1+\tilde{s}p}}.$$
Then by using the above estimates in \eqref{2-3}, we obtain, 
\begin{equation}\label{est2}
\begin{split}
\mathcal P_{2,3}(x,\epsilon) 
&\geq   \int_0^{1-\frac{|\epsilon|}{x + \lambda^{1/\alpha}}}  \frac{(1 - y^{\tilde{s}})^{p-1} (1-y^{\tilde{s}-1}) }{(1-y)^{1+\tilde{s}p}} ~dy \\
&=  \left[ \frac{1}{\tilde{s}p} \frac{(1-y^{\tilde{s}})^p}{(1-y)^{\tilde{s}p}} \right]_0^{1-\frac{|\epsilon|}{x + \lambda^{1/\alpha}}}
 =\frac{1}{\tilde{s} p} \left(\left( \frac{(x+ \lambda^{1/\alpha})^{\tilde{s}}- (x+ \lambda^{1/\alpha}-|\epsilon|)^{\tilde{s}}}{|\epsilon|^{\tilde{s}}}\right)^p-1  \right)\\
 &\geq -\frac{1}{\tilde sp}. 
\end{split}
\end{equation}
\textbf{Case 2:}  $\beta\geq 1$\\
In the same way, we note that $\mathcal P_{2,3}(x,\epsilon)\geq 0$. Now, for the upper bound, using $1-y^\kappa \leq \max\{1,\kappa\}(1-y)$ for any $y\in(0,1)$ and $\kappa>0$  we get:
\begin{equation}\label{est201}
\mathcal P_{2,3}(x,\epsilon)
 \leq \max\{1,\beta-1\} \int_0^{1-\frac{|\epsilon|}{x + \lambda^{1/\alpha}}} (1-y)^{p(1-s)-1} ~dy
\leq \frac{\max\{1,\beta-1\}}{p(1-s)}.
\end{equation}
Finally we estimate the last term $\mathcal P_4$:\\
\begin{equation}\label{est4}
\begin{split}
|\mathcal P_4(x,\epsilon)| &\leq \int_{1+\frac{|\epsilon|}{x + \lambda^{1/\alpha}}}^{1+\frac{|\epsilon|}{x + \lambda^{1/\alpha}- |\epsilon|}} \frac{| y^\alpha -1|^{p-1}}{|1-y|^{1+sp}} ~dy \leq \int_{1+\frac{|\epsilon|}{x + \lambda^{1/\alpha}}}^{1+\frac{|\epsilon|}{x + \lambda^{1/\alpha}- |\epsilon|}} \frac{| y^s -1|^{p-1}}{|y-1|^{1+sp}} ~dy\\
& \leq \int_{1+\frac{|\epsilon|}{x + \lambda^{1/\alpha}}}^{1+\frac{|\epsilon|}{x + \lambda^{1/\alpha}- |\epsilon|}} \frac{| y-1|^{s(p-1)}}{|y-1|^{1+sp}} ~dy = \frac{1}{s} \frac{(x+\lambda^{1/\alpha})^s- (x+ \lambda^{1/\alpha}-|\epsilon|)^s}{|\epsilon|^s}:= \frac{\xi_\epsilon(x)}{s}.
\end{split}
\end{equation}
Noting $\xi_\epsilon(x)\to 0$ {\it a.e.} in $x\in \R_+$, we deduce, combining \eqref{est01}-\eqref{est4}, that there exist two constants $C_1$ and $C_2$ independent of $x$  such that, for any $x\in \R_+$: 
$$ C_1 \leq \lim_{\epsilon\to 0}\mathcal{P}_\epsilon(x) \leq C_2.$$
Hence we deduce \eqref{est1}. More precisely, the constant $C_1$ and $C_2$ are given by
\begin{equation}\label{c1}
 \begin{aligned}
 C_1=  \left\{
 \begin{array}{ll}
  \frac{1}{p}\left(\frac{\tilde{s}-s}{\tilde{s}s}\right)  &  \text{ if } \beta<1, \\
 \frac{1}{sp} & \text{ if } \beta \geq 1 ,\\
 \end{array} 
 \right. \quad \mbox{ and } \quad
 C_2=  \left\{
 \begin{array}{ll}
  \frac{1}{sp}
  &  \text{ if } \beta<1, \\
 \frac{1}{sp}+ \frac{\max\{ 1, \beta-1  \}}{p(1-s)} & \text{ if } \beta \geq 1.\\
 \end{array} 
 \right.
 \end{aligned}
 \end{equation}
\noindent Finally the assertion, $U_\lambda \in \overline{W}^{s,p}_{loc}(\mathbb{R}_+)$ follows by showing $U_\lambda \in \overline{W}^{s,p}(a,b)$ for all $-\lambda^{1/\alpha} < a < b < \infty.$
Indeed, using the symmetry of the integrand and changes of variable, we obtain 
\begin{equation}\label{sp}
\begin{split}
\iint_{[a,b]^2} \frac{|U_\lambda(x)-U_\lambda(y)|^{p}}{|x-y|^{1+sp}} ~dx~dy 
& = \int_{a+\lambda^{1/\alpha}}^{b+\lambda^{1/\alpha}} \int_{a+\lambda^{1/\alpha}}^{b+\lambda^{1/\alpha}} \frac{|x^\alpha - y^\alpha|^p}{|x-y|^{1+sp}} ~dx~dy\\
&= 2 \int_{a+\lambda^{1/\alpha}}^{b+\lambda^{1/\alpha}} \int_{a+\lambda^{1/\alpha}}^{x} \frac{|x^\alpha - y^\alpha|^p}{|x-y|^{1+sp}} ~dy~dx\\
& = 2 \int_{a+\lambda^{1/\alpha}}^{b+\lambda^{1/\alpha}} x^{\alpha p-sp} \int_{\frac{a+\lambda^{1/\alpha}}{x}}^{1} \frac{(1-t^{\alpha})^p}{(1-t)^{1+sp}} ~dt ~dx \\
&< 2 \int_{a+\lambda^{1/\alpha}}^{b+\lambda^{1/\alpha}} x^{\alpha p-sp} \int_{0}^{1} \frac{(1-t)^p}{(1-t)^{1+sp}} ~dt ~dx < \infty
\end{split}
\end{equation}
for any $\alpha\in (0,s)$ if $\lambda>0$ and $\alpha\in (s-\frac{1}{p} ,s)$ if $\lambda=0$. 
\end{proof} 
\noindent Next, we study the behavior of $(-\Delta)^s_p V_\lambda(x)$ on $\mathbb{R}^N_+:= \{x \in \mathbb{R}^N: x_N > 0\}$ where $V_\lambda(x):= U_\lambda(x\cdot e_N)=U_\lambda(x_N)$.\\
Let $GL_N$ be the set of $N \times N$ invertible matrices
, we have
\begin{cor}\label{esthm}
Let $\lambda \geq 0$, $\alpha\in(0,s)$, $A\in  GL_N$ and $p>1$. Let $\mathcal J_{\epsilon,A}$ be the function defined on $\R^N_+$ by
$$\mathcal J_{\epsilon, A}(x) = \int_{(B_\epsilon(0))^c} \frac{[V_\lambda(x)-V_\lambda(x+z)]^{p-1}}{|Az|^{N+sp}} ~dz$$
for some $\epsilon>0$.\\
Then, there exist two positive constants $C_3$ and $C_4$ depending on $\alpha,s,p,N,\|A\|_2, \|A^{-1}\|_2$ such that
\begin{equation}\label{est5}
C_3 (x_N+ \lambda^{1/\alpha})^{-\beta} \leq \lim_{\epsilon \to 0} \mathcal J_{\epsilon,A}(x) \leq  C_4 (x_N+ \lambda^{1/\alpha})^{-\beta} 
\end{equation}
pointwisely in $\mathbb{R}^N_+ \times GL_N.$ 
Moreover,  for $\lambda>0$, $V_\lambda \in \overline{W}^{s,p}_{loc}(\mathbb{R}^N_+)$ and for $\lambda=0$, $V_\lambda \in \overline{W}^{s,p}_{loc}(\mathbb{R}^N_+)$ if $s-\frac{1}{p} < \alpha <s$.
\end{cor} 
\begin{proof}
As in the proof of Lemma 3.2 in \cite{IMS}, we define the elliptic coordinates for any $y\in \R^N\setminus \{0\}$ as $y= \rho w$ where  $\rho >0$ and $w \in \mathcal E:=AS^{N-1}$. Hence we have $dy= \rho^{N-1} d\rho dw$ where $dw$ is the surface of $\mathcal E$.
We also define $e_A= ^t\!\!(A^{-1}) e_N$ and $E_A= \{x \in \mathbb{R}^N: x \cdot e_A > 0\}$ then we have
$$e_A\cdot w = (A^{-1} w)_N,\  \ \forall \ w \in \mathcal E .$$ 
Let $x\in \R^N_+$, by the change of variable $z= \rho A^{-1}  w$:
\begin{equation*}
\begin{split}
\mathcal J_{\epsilon, A}(x)
&= |\det A|^{-1} \int_{\mathcal E} \frac{1}{|w|^{N+sp}} \int_\epsilon^\infty \frac{[U_\lambda(x_N)-U_\lambda(x_N+ \rho(e_A\cdot w) )]^{p-1}}{|\rho|^{1+sp}} ~d\rho ~dw \\
& = |\det A|^{-1} \left( \int_{\mathcal E \cap E_A} \int_\epsilon^\infty + \int_{\mathcal E\cap (E_A)^c} \int_\epsilon^\infty  \right).
\end{split}
\end{equation*}
Replacing $\rho$ and $w$ by $-\rho$ and $-w$ in the second integral in the right-hand side and noting $-w \in \mathcal E \cap E_A$, we get
\begin{equation*}
\mathcal J_{\epsilon, A}(x)= |\det  A|^{-1} \int_{\mathcal E \cap E_A} \frac{1}{|w|^{N+sp}}\int_{(-\epsilon, \epsilon)^c} \frac{[U_\lambda(x_N)-U_\lambda(x_N+ \rho(e_A\cdot w) )]^{p-1}}{|\rho|^{1+sp}} ~d\rho ~dw . 
\end{equation*}
Now, the new change of variable $t=x_N +\rho(e_A \cdot w)$ yields in $\mathcal J_{\epsilon,A}$:
\begin{equation*}
\begin{split}
\mathcal J_{\epsilon, A}(x)
&= (x_N+ \lambda^{1/\alpha})^{-\beta}|\det A|^{-1} \int_{\mathcal E \cap E_A} \frac{|e_A \cdot w|^{sp}}{|w|^{N+sp}} \mathcal{P}_{(e_A \cdot w) \epsilon}(x_N) ~dw.
\end{split}
\end{equation*}
Noting that
\begin{equation}\label{est6}
|\det A|^{-1} \int_{\mathcal E \cap E_A} \frac{|e_A \cdot w|^{sp}}{|w|^{N+sp}} dw= \frac{1}{2} \int_{S^{N-1}} \frac{|e_N \cdot v|^{sp}}{|Av|^{N+sp}} dv<\infty,
\end{equation} 
we obtain \eqref{est5} passing to the limit $\epsilon\to 0$ and using Theorem \ref{thest}.\\
Finally, the assertion $V_\lambda \in \overline{W}^{s,p}_{loc}(\mathbb{R}^N_+)$ follows showing $V_\lambda \in \overline{W}^{s,p}(K)$ for any bounded set $K \Subset \mathbb{R}^N_+$
and using the computations in \eqref{sp}.
\end{proof}
\noindent For the next estimate, we prove the following lemma:
\begin{Lem}\label{calint}
Let $r \in \mathbb{R}_+$ and for any $\vartheta\in (0,\min\{1,r\})$, we define
$$\mathcal H_\vartheta(r)=\int_{ (-\vartheta, \vartheta)^c\cap (-1,1)}  \frac{|U_\lambda(r)-U_\lambda(r+ t)|^{p-1}}{|
t|^{sp}} ~dt.$$
Then, there exists a constant $C>0$ depending upon $\alpha,\,s$ and $p$ such that for any $r>0$ and $\vartheta\in (0,\min\{1,r\})$
\begin{equation}\label{est-new8}
  \mathcal H_{\vartheta}(r) \leq C (r + \lambda^{1/\alpha})^{-\beta}((r + \lambda^{1/\alpha})^s+  (r + \lambda^{1/\alpha})+(r + \lambda^{1/\alpha})^{\beta}).
\end{equation}
\end{Lem}
\begin{proof}
As previously, to estimate $\mathcal H_{\vartheta}$, we split the integral as follows 
\begin{equation}\label{est-new4}
    \begin{split}
        \mathcal H_{\vartheta}(r) &=   \int_{r-1}
        ^{r-\vartheta} \frac{|U_\lambda(r)-U_\lambda(t)|^{p-1}}{|r-t|^{sp}} ~dt 
        +  \int_{r + \vartheta}^{r+1}  \frac{|U_\lambda(r)-U_\lambda(t)|^{p-1}}{|r-t|^{sp}} ~dt\\
        &= \mathcal H_{\vartheta,1}(r) + \mathcal H_{\vartheta,2}(r). 
    \end{split}
\end{equation}
For $\mathcal H_{\vartheta,1}$, we consider two cases: for $r\leq 1-\lambda^{1/\alpha}$, we have \begin{equation*}
    \mathcal H_{\vartheta,1}(r)= \int_{r-1}^{-\lambda^{1/\alpha}} \frac{|U_\lambda(r)|^{p-1}}{|r-t|^{sp}} ~dt+\int_{-\lambda^{1/\alpha}}         ^{r-\vartheta} \frac{|U_\lambda(r)-U_\lambda(t)|^{p-1}}{|r-t|^{sp}} ~dt .
\end{equation*}
Hence the first term in the right-hand side is bounded by 
\begin{equation}\label{est-new5}
    \left\{
    \begin{array}{l l}
       \frac{1}{sp-1}(r + \lambda^{1/\alpha})^{1-\beta} &\text{if} \ sp >1,\\
       C(\alpha,s,p)
       &\text{if} \ sp \leq 1.\\
    \end{array}
    \right.
\end{equation}
Using a change of variable in the second term of the right-hand side and for any $t\in(0,1)$,\break $1-t^\alpha\leq1-t^s\leq (1-t)^s$, we get
\begin{equation}\label{est-new6}
    \begin{split}
       (r + \lambda^{1/\alpha})^{1-\beta} \int_{0}^{1- \frac{\vartheta}{r+ \lambda^{1/\alpha}}} \frac{(1-t^{\alpha})^{p-1}}{(1-t)^{sp}} ~dt &\leq (r + \lambda^{1/\alpha})^{1-\beta} \int_{0}^{1- \frac{\vartheta}{r+ \lambda^{1/\alpha}}}  (1-t)^{-s} ~dt\\
       &\leq \frac{1}{1-s}(r + \lambda^{1/\alpha})^{1-\beta}.
    \end{split}
\end{equation}
For $r> 1-\lambda^{1/\alpha}$, we have
\begin{equation*}
    \mathcal H_{\vartheta,1}(r)\leq \int_{-\lambda^{1/\alpha}}^{r-\vartheta} \frac{|U_\lambda(r)-U_\lambda(t)|^{p-1}}{|r-t|^{sp}} ~dt\leq \frac{1}{1-s}(r + \lambda^{1/\alpha})^{1-\beta}.
\end{equation*}
In the same way for $\mathcal H_{\vartheta,2}$, since for any $t\geq 1$, $t^\alpha-1\leq t^s-1\leq (t-1)^s$, we get:
\begin{equation}\label{est-new7}
    \begin{split}
       \mathcal H_{\vartheta,2}(r) &\leq (r + \lambda^{1/\alpha})^{1-\beta} \int_{1+ \frac{\vartheta}{r + \lambda^{1/\alpha}}}^ {1+ \frac{1}{r + \lambda^{1/\alpha}}} \frac{(t^{\alpha}-1)^{p-1}}{(t-1)^{sp}} ~dt\\
       & \leq (r+ \lambda^{1/\alpha})^{1-\beta} \int_{1+ \frac{\vartheta}{r + \lambda^{1/\alpha}}}^ {1+ \frac{1}{r + \lambda^{1/\alpha}}} (t-1)^{-s} ~dt \leq \frac{1}{1-s}(r + \lambda^{1/\alpha})^{s-\beta}.
    \end{split}
\end{equation}
Then, by collecting the estimates \eqref{est-new5}-\eqref{est-new7}, we obtain  \eqref{est-new8}.
\end{proof}
\noindent Now the next result gives the corresponding estimates of $(-\Delta)^s_p \left((x_N+ \lambda^{1/\alpha})^+\right)^\alpha$ under the smooth change of coordinates. 
\begin{thm}\label{thpr}
Let $\alpha\in(0,s)$ and $p>1$. Let $\psi :\mathbb{R}^N \to \mathbb{R}^N$ be a  $C^{1,1}$-diffeomorphism 
in $\R^N$ such that $\psi= Id$ in $B_R(0)^c$, for some $R>0$. \\
Then, considering  $W_\lambda(x)=U_\lambda(\psi^{-1}(x) \cdot e_N)$, there exist $\rho^*=\rho^*(\psi)>0$ and $\lambda^*=\lambda^*(\psi)>0$ such that for any $\rho\in(0,\rho^*)$, there exists a constant $\tilde{C}>0$ independent of $\lambda$ such that, for any $\lambda\in {[0,\lambda^*]}$,
\begin{equation}\label{core}
\frac{1}{\tilde C} W_\lambda(x)^{-\frac\beta\alpha} \leq (-\Delta)^s_p W_\lambda(x) \leq \tilde{C} W_\lambda(x)^{-\frac\beta\alpha} \quad  \mbox{E-weakly in $\psi(\{X:0< X_N < \rho\})$}.
\end{equation}
\end{thm}
\begin{proof}
Define, for any $x\in \psi(\R^N_+)$, $H(x)=2\lim_{\epsilon \to 0} H_\epsilon (x)$ where for $\epsilon>0$,  \begin{equation}\label{est701}
H_\epsilon(x)= \int_{D_\epsilon(x)} \frac{[W_\lambda(x)-W_\lambda(y)]^{p-1}}{|x-y|^{N+sp}} ~dy
\end{equation}
and $D_\epsilon(x)=  \{y\in \R^N:\ |\psi^{-1}(x)-\psi^{-1}(y)|>\epsilon \}$. \\
By change of variable, with the notations $x=\psi(X)$ and $A_X=D\psi(X)$,  we have:
\begin{equation*}
\begin{split}
    H_\epsilon(x)&= |\det A_X|\mathcal J_{\epsilon, A_X}(X)+\int_{(B_\epsilon(X))^c}\frac{[U_\lambda(X_N)-U_\lambda(Y_N)]^{p-1}}{|A_X(X-Y)|^{N+ps}}h(X,Y)~dY\\
    &= H_{\epsilon,1}(X) + H_{\epsilon,2}(X) 
    \end{split}
\end{equation*}
where, by Lemma 3.4 in \cite{IMS}, there exists a constant $C_\psi$ such that
\begin{equation*}
    \begin{split}
        |h(X,Y)|&=\left| \frac{|A_X(X-Y)|^{N+ps}}{|\psi(X)-\psi(Y)|^{N+sp}}|\det A_Y|-|\det A_X|\right|\\
        &\leq C_\psi\min\{|X-Y|,1\}.
    \end{split}
\end{equation*}
In order to apply Lemma \ref{equi:stro-weak}, first we prove uniform estimates of $H_\epsilon$ on compact set of $\psi (\R^N_+)$.
Since $\psi$ is a $C^{1,1}-$ diffeomorphism such that  $\psi= Id$ in $B_R(0)^c$ for some $R>0$  therefore the mappings $X\mapsto  |\det D\psi(X))|$ and $X\mapsto  \| D\psi(X))\|_\infty$ are bounded on $\R^N$. More precisely, there exists a constant $c_\psi>0$ such that for any $X \in \mathbb{R}^N$
\begin{equation}\label{est8}
\frac{1}{c_\psi} \leq |det D\psi(X)| \leq c_\psi \quad \mbox{ and } \quad \frac{1}{c_\psi} \leq \| D\psi(X))\|_\infty \leq c_\psi.
\end{equation}
Hence plugging \eqref{est6} and \eqref{est8}, we obtain $H_{\epsilon,1}$ is bounded in $\R^N$.
Now, we give an estimate of $ H_{\epsilon,2}$ {in $\{X \in \mathbb{R}^N_+ :0<X_N<1\}$}: 
\begin{equation}\label{est-new1}
\begin{split}
| H_{\epsilon,2}(X)| 
\leq&\ C_\psi \bigg( \int_{B_1(X) \setminus B_\epsilon(X) } \frac{|U_\lambda(X_N)-U_\lambda(Y_N)|^{p-1} |X-Y|}{|
A_X(X-Y)|^{N+sp}} ~dY \\
& + \int_{(B_1(X))^c} \frac{|U_\lambda(X_N)-U_\lambda(Y_N)|^{p-1}}{|A_X(X-Y)|^{N+sp}} ~dY  \bigg)\\
=&\ C_\psi\left( H_{\epsilon,2}^{\star}(X) + H_{\epsilon,2}^{\diamond}(X) \right).
\end{split}
\end{equation}
First, by H\"older regularity of the mapping $x\mapsto x^\alpha$, we have for any $X\in \R^N_+$:
\begin{equation}\label{est-new2}
    H_{\epsilon,2}^{\diamond}(X) \leq  C_\psi\int_{1}^\infty  \frac{1}{t^{1+\beta}} ~dt \leq C_\psi.
\end{equation}
For the first term, using  polar coordinates $Y =X+ \sigma w$ for $w \in S^{N-1}$, $\sigma >0$, {$X \in \mathbb{R}^N_+$ and by choosing $\epsilon <X_N$}, we obtain from \eqref{est8}
\begin{equation}\label{est-new3}
\begin{split}
    H_{\epsilon,2}^{\star}(X) 
&\leq  c_\psi \int_{S^{N-1}} \frac{1}{|w|^{N+sp-1}} \int_{\epsilon}^1  \frac{|U_\lambda(X_N)-U_\lambda(X_N+ \sigma w_N)|^{p-1}}{|
\sigma|^{sp}} ~d\sigma ~dw\\
&=  c_\psi \int_{S^{N-1} \cap \{w_N >0\}} |w|^{-N} \int_{ (-\epsilon w_N, \epsilon w_N)^c \cap(-w_N,w_N)}  \frac{|U_\lambda(X_N)-U_\lambda(X_N+ t)|^{p-1}}{|
t|^{sp}} ~dt ~dw\\
& \leq  c_\psi \int_{S^{N-1} \cap \{w_N >0\}} \int_{ (-\epsilon w_N, \epsilon w_N)^c\cap (-1,1)}  \frac{|U_\lambda(X_N)-U_\lambda(X_N+ t)|^{p-1}}{|
t|^{sp}} ~dt ~dw\\
& =  c_\psi \int_{S^{N-1} \cap \{w_N >0\}} \mathcal {H}_{\epsilon w_N}(X_N) ~dw
\end{split}
\end{equation}
where the function $\mathcal H_{\epsilon w_N}$ is defined in Lemma \ref{calint}.\\
From \eqref{est-new8}, we deduce that $ H^\star_{\epsilon,2}$ and thus $H_{\epsilon,2}$ are bounded on compact sets of $\R^N_+$. Hence, $H_\epsilon$ converges to $\frac12 H$ in $L^1_{loc}(\psi(\R^N_+))$ and we apply Lemma \ref{equi:stro-weak} which implies that $W_\lambda$ satisfies  $(-\Delta)_p^s W_\lambda= H$ E-weakly in $\psi(\R^N_+)$.\\
Furthermore, since \eqref{est-new8} is independent of $\vartheta$, then gathering \eqref{est-new2}, \eqref{est-new3}, \eqref{est-new8} in \eqref{est-new1}, there exist $\lambda^*$ and $\rho^*$ small enough, for any $\lambda\leq\lambda^*$ and $\rho\leq \rho^*$, there exists a constant $\tilde C$ independent of $\lambda$ and $\epsilon$ such that for any $X \in \{X:0< X_N < \rho\}$:
\begin{equation}\label{est-new9}
  |H_{\epsilon,2}(X)| \leq \tilde{C} (1+ (X_N + \lambda^{1/\alpha})^{s-\beta}) \leq \frac{C_3}{2c_\psi}  (X_N + \lambda^{1/\alpha})^{-\beta}
\end{equation}
where $C_3$ is defined in \eqref{est5}.\\
Finally, by combining \eqref{est5}, \eqref{est8} and \eqref{est-new9}, there exists a constant   $\tilde{C}$ independent of $\lambda$ such that \begin{equation}\label{est-new0}
\frac{1}{\tilde{C}} (X_N + \lambda^{1/\alpha})^{-\beta} \leq \lim_{\epsilon\to0} H_\epsilon(x) \leq \tilde{C} (X_N + \lambda^{1/\alpha})^{-\beta}, \quad  \forall \ x \in \psi(\{X:0< X_N < \rho\})
\end{equation}
and we deduce \eqref{core}.
\end{proof}
\noindent We extend the definition of the function $d$ in $\Omega^c=\R^N\setminus \Omega$ as follows
\begin{equation}\label{ext:def}
 \begin{aligned}
d_e(x)=  \left\{
 \begin{array}{ll}
 \dist(x,\partial \Omega) & \mbox{ if } \ x\in \Omega;\\
- \dist(x, \partial\Omega)  &  \text{ if } \ x \in (\Omega^c)_{\lambda^{\frac1\alpha}};\\
 -\lambda^{1/\alpha} & \text{ otherwise}.\\
 \end{array} 
 \right.
 \end{aligned}
 \end{equation}
 where $(\Omega^c)_{\eta} = \{ x \in \Omega^c\ :\ \dist(x,\partial\Omega) < \eta \}$.
 Hence we define, for some $\rho>0$ and $\lambda >0$:
\begin{equation}\label{sub}
 \begin{aligned}
\underline w_\rho(x)=  \left\{
 \begin{array}{ll}
 (d_e(x)+ \lambda^{1/\alpha})_+^{\alpha}- \lambda  &  \text{if } \ x \in \Omega \cup (\Omega^c)_{\rho},\\
 { - \lambda} & \text{ otherwise} ,\\
 \end{array} 
 \right.
 \end{aligned}
 \end{equation}
 \begin{equation}\label{super}
 \begin{aligned}
 \overline w_\rho(x)=  \left\{
 \begin{array}{ll}
 (d_e(x)+ \lambda^{1/\alpha})_+^{\alpha}  &  \text{if } x \in \Omega \cup (\Omega^c)_{\rho}, \\
 0 & \text{otherwise}.\\
 \end{array} 
 \right.
 \end{aligned}
 \end{equation}

\begin{thm}\label{thq}
Let $\Omega \subset \mathbb{R}^N$ be a smooth bounded domain with a $C^{1,1}$ boundary and $ \alpha\in (0,s)$. Then, for some $\rho>0$, there exist $(\lambda_*,\eta_*)\in \R_*^+\times \R_*^+$ such that for any $\eta<\eta_*$, there exist positive constants $C_5, C_6$ 
such that for any $\lambda\in [0,\lambda_*]$:
\begin{equation}\label{mainest}
  (-\Delta)^s_p \overline w_\rho \geq C_5 (d(x) + \lambda^{1/\alpha})^{-\beta} \ \text{ and }\ (-\Delta)^s_p\underline w_\rho \leq C_6 (d(x) + \lambda^{1/\alpha})^{-\beta}\  \text{E-weakly in } \ \Omega_{{\eta}}
\end{equation}
where $\Omega_\eta = \{x \in \Omega: d(x) < \eta\}$. Moreover, for $\lambda >0$, $\underline w_\rho$, $\overline w_\rho$ belong to $\overline W^{s,p}(\Omega_\eta)$. 
\end{thm}
\begin{proof}
Since $\partial \Omega \in C^{1,1}$, then for every $x \in \partial \Omega$, there exist a neighbourhood $N_x$ of $x$ and a bijective map $\Psi_x:Q\mapsto N_x$ such that 
$$ \Psi_x \in C^{1,1}(\overline{Q}), \ \Psi_x^{-1} \in C^{1,1}(\overline{N_x}),\ \Psi_x(Q_+)= N_x \cap \Omega \ \text{ and } \ \Psi_x(Q_0)= N_x \cap \partial\Omega$$
where $Q:= \{X= (X',X_N): |X'| <1,\ |X_N| <1\},\ Q_+ := Q \cap \mathbb{R}^N_+,\ Q_0 := Q \cap \{X_N=0\}$. \\

For any $x \in \partial \Omega$, $0<\tilde \rho< \rho< \rho^*$ where $\rho^*$ is defined in Theorem \ref{thpr} and using the fact that $\partial \Omega$ is compact, there exist a finite covering $\{B_{R_i}(x_i)\}_{i \in I}$ of $\partial \Omega$ and $\eta^*= \eta^*(R_i)$, $i \in I$ such that for any $\eta \in (0, \eta^*)$
\begin{equation}\label{cover:omega}
 \Omega_\eta\subset \bigcup_{i\in I} B_{R_i}(x_i) \quad \mbox{ and } \ \forall i\in I, \quad \Psi^{-1}_{x_i}(B_{R_i}(x_i))\subset B_{\tilde\rho}(0)\subset B_{\rho}(0).   
\end{equation}
Now by using the geometry of $\partial\Omega$ and arguing as in Lemma $3.5$ and Theorem $3.6$ in \cite{IMS}, there exist diffeomorphisms $\Phi_i \in C^{1,1}(\mathbb{R}^N, \mathbb{R}^N)$  for any $i\in I$ satisfying $\Phi_i= \Psi_{x_i}$ in $B_{\rho}(0)$ and $\Phi_i=Id$ in   $(B_{4\rho}(0))^c$, 
$$\Omega_{\eta} \cap B_{R_i}(x_i)\Subset \Phi_i(B_{\tilde\rho} \cap \mathbb{R}^N_+),\quad 
d_e(\Phi_i(X))= (X_N + \lambda^{1/\alpha})_+ - \lambda^{1/\alpha}, \quad \forall \ X \in B_{\rho} $$
and for $\lambda$ small enough $\lambda^{1/\alpha}<\rho$, 
$$\Phi_i(B_\rho(0)\cap \{X_N\geq -\lambda^{1/\alpha}\}) \subset \Omega\cup (\Omega^c)_\rho.$$
Using the finite covering, it is sufficient to prove the statement in any of set {$\Omega_\eta \cap B_{R_i}(x_i)$} with $x_i \in \partial\Omega$ and  for the sake of simplicity we can suppose $x_i=0$, $\Phi_i= \Phi$ and $\Phi(0)=0$.\\ 
Let $g_{\epsilon,1}$ and $g_{\epsilon,2}$ be two functions defined by
 $$g_{\epsilon,1}(x)=\int_{D_\epsilon(x)}\frac{[\underline w_\rho(x)-\underline w_\rho(y)]^{p-1}}{|x-y|^{N+sp}} ~dy$$
and $$g_{\epsilon,2}(x)=\int_{D_\epsilon(x)}\frac{[\overline w_\rho(x)- \overline w_\rho(y)]^{p-1}}{|x-y|^{N+sp}} ~dy$$ 
where $D_\epsilon(x)= \{y: |\Phi^{-1}(x)-\Phi^{-1}(y)|> \epsilon\}$.\\
As in the proof of Theorem \ref{thpr}, it suffices to obtain suitable uniform bounds on compact sets of $g_{\epsilon,1}$ and $g_{\epsilon,2}$. Hence Lemma \ref{equi:stro-weak} gives estimates \eqref{mainest}.\\
Let {$x \in B_{R_i}(0) \cap \Omega_{\eta}$}, there exists $X \in B_{\tilde\rho}(0) \cap \mathbb{R}^N_+$ such that $\Phi(X)=x$ and hence by change of variables {and arguing as in Theorem $3.6$ in \cite{IMS}}, we obtain
\begin{equation*}
\begin{split}
g_{\epsilon,1}(x)=& \int_{(B_\epsilon(X))^c}\frac{[\underline{w}_\rho(\Phi(X))-\underline{w}_\rho(\Phi(Y))]^{p-1}}{|\Phi(X)-\Phi(Y)|^{N+sp}}  |\det D\Phi(Y)|~dY \\
=& \int_{B_\rho(0) \setminus B_{\epsilon}(X)}\dots + \int_{(B_{\rho}(0))^c}\dots \\
=& \int_{(B_{\epsilon}(X))^c} \frac{[(X_N + \lambda^{1/\alpha})^\alpha_+-(Y_N + \lambda^{1/\alpha})^\alpha_+]^{p-1}}{|\Phi(X)-\Phi(Y)|^{N+sp}} |\det D\Phi(Y)| dY\\
&+ \int_{(B_{\rho}(0))^c)} \frac{[\underline{w}_\rho(\Phi(X)-\underline{w}_\rho(\Phi(Y)))]^{p-1}- [U_\lambda(X_N)-U_\lambda(Y_N)]^{p-1} }{|\Phi(X)-\Phi(Y)|^{N+sp}} |\det D\Phi(Y)| dY\\
=& M_\epsilon(X) + M_{\underline w_\rho}(X) \end{split}
\end{equation*}
and similarly, 
\begin{equation*}
\begin{split}
g_{\epsilon,2}(x)&= \int_{(B_\epsilon(X))^c}\frac{[\overline w_\rho(\Phi(X))-\overline w_\rho(\Phi(Y))]^{p-1}}{|\Phi(X)-\Phi(Y)|^{N+sp}}  |\det D\Phi(Y)| ~dY \\
& = \int_{B_\epsilon^c(X)} \frac{[(X_N + \lambda^{1/\alpha})^\alpha_+-(Y_N + \lambda^{1/\alpha})^\alpha_+]^{p-1}}{|\Phi(X)-\Phi(Y)|^{N+sp}} |\det D\Phi(Y)| dY\\
& \quad + \int_{(B_{\rho}(0))^c} \frac{[\overline w_\rho(\Phi(X))-\overline w_\rho(\Phi(Y))]^{p-1}- [U_\lambda(X_N)-U_\lambda(Y_N)]^{p-1} }{|\Phi(X)-\Phi(Y)|^{N+sp}} |\det D\Phi(Y)| dY\\
&= M_\epsilon(X) + M_{\overline w_\rho}(X).
\end{split}
\end{equation*}
From the Lipschitz continuity of $\Phi^{-1}$, the $\alpha$-H\"older continuity of $U_\lambda$, $\underline{w}_\rho$ and $\overline w_\rho$, we obtain by using \eqref{est8} for $w=\underline w_\rho$ or $w=\overline w_\rho$:
\begin{equation}\label{est12}
|M_w(X)|
\leq c_\Phi   \int_{(B_{\rho}(0))^c} \frac{2}{|X-Y|^{N+\beta}} dY\leq C(\Phi, \rho, \tilde\rho) \int_{\mathbb{R}^N}  \frac{1}{(1+|Y|)^{N+\beta}} dY \leq \mathscr{C}
\end{equation}
where $\mathscr C$ is a constant independent of $X,\, \lambda$ and $\epsilon$.\\
{Now we deal with }$M_\epsilon$ performing change of variables. We note $M_\epsilon$ coincides with $H_\epsilon$ in \eqref{est701}. Hence, using the estimate in \eqref{est-new0}, we get 
\begin{equation}\label{est13}
c_3 (d(x)+\lambda^{1/\alpha})^{-\beta} \leq \lim_{\epsilon \to 0} M_\epsilon(\Phi^{-1}(x)) \leq c_4 (d(x)+\lambda^{1/\alpha})^{-\beta} \ \ \text{E-weakly in} \ {\Omega_{\eta} \cap B_{R_i}(0)}.
\end{equation}
where  $c_3$ and $c_4$ are positive constant depending upon $\alpha,\,N,\,s,\,p$ and $\Phi.$
By combining \eqref{est12} and \eqref{est13} for any $i \in I$, we obtain for all $x \in \Omega_{\eta}$
$$(-\Delta)^s_p\underline w_\rho(x) \leq c_3 (d(x)+\lambda^{1/\alpha})^{-\beta} + \mathscr{C} \ \ \text{E-weakly in} \ \Omega_{\eta}$$
and 
$$c_4 (d(x)+\lambda^{1/\alpha})^{-\beta} -   \mathscr{C}  \leq (-\Delta)^s_p \overline w_\rho(x) \ \ \text{E-weakly in} \ \Omega_{\eta}.$$
Finally, we deduce the estimates \eqref{mainest} taking $\eta$ and $\lambda$ small enough.
 \\
To prove $\underline w_\rho,  \overline w_\rho \in \overline{W}^{s,p}(\Omega_\eta)$ for $\lambda >0$, it is sufficient to claim
\begin{equation*}
\underline w_\rho,  \overline w_\rho \in W^{s,p}(K),\quad K:=\Omega_{\eta_1} \cup (\Omega^c)_{\eta_2} 
\end{equation*}
for some $0< \eta < \eta_1$ and $\eta_2>0.$ \\
For $x_i \in \partial \Omega$, {for $\eta_0\in (0,\eta^*)$}, let $\{B_{R_i}(x_i)\}_{i \in I}$ be the finite covering of $\Omega_{\eta_0}$ and  $\Xi_{i} \in C^{1,1}(\mathbb{R}^N, \mathbb{R}^N)$ such that
\begin{equation}\label{diff:pro}
\begin{split}
B_{R_i}(x_i) &\Subset \Xi_i(B_{\xi_0}),\quad  d_e(\Xi_i(X))= (X_N + \lambda^{1/\alpha})_+ - \lambda^{1/\alpha}, \quad \forall \ X \in B_{\xi_0}
\end{split}
\end{equation}
for some $\xi_0\in (0,\frac{\lambda^{1/\alpha}}{2}).$ The existence of finite covering $\{B_{R_i}(x_i)\}_{i \in I}$ and diffeomorphisms $\Xi_i$ are obtained as above by using \eqref{cover:omega} .\\
For any $i \in I$, there exists a subset $J^i$ of $I$ such that 
\begin{equation}\label{def:adjacent}
  B_{R_i}(x_i) \cap B_{R_j}(x_j) \neq \emptyset \ \ \forall \ j \in J^i.  
\end{equation}
The collection of sets $\{B_{R_j}(x_j)\}_{j \in J^i}$ satisfying \eqref{def:adjacent} are called adjacent sets to $B_{R_i}(x_i).$ Now for any $i \in I$ and $j \in J^i$, define for some $\tau_i < R_i$
\begin{equation}\label{est:insidecover}
K_i:= B_{\tau_i} (x_i) \subset B_{R_i}(x_i) 
\end{equation}
such that 
\begin{equation}\label{est:newcover}
 {\mbox{for any } i\in I,\   K_i \cap K_j \neq \emptyset \ \ \forall \ j \in J_i} \quad \text{and} \quad  \min_{i \in I}\big( \min_{j \in J^i}\, \dist(K_j \setminus B_{R_i}(x_i), K_i)\big) >0.
\end{equation}
{By using \eqref{est:insidecover} and \eqref{est:newcover}, we choose $\eta_1$ and $\eta_2$ small enough such that} 
$$ K = \Omega_{\eta_1} \cup (\Omega^c)_{\eta_2} \subset \bigcup_{i \in I} K_i.$$
Now by using \eqref{diff:pro}, we obtain, for any $i\in I$
\begin{equation}\label{diff:newpro}
\begin{array}{l}
\Omega_{\eta_1} \cap  K_i \subset \Omega_{\eta_1} \cap  B_{R_i}(x_i) \Subset \Xi_i(B_{\xi_0} \cap  \mathbb{R}^N_+),  \\
(\Omega^c)_{\eta_2} \cap K_i \subset (\Omega^c)_{\eta_2} \cap  B_{R_i}(x_i) \Subset \Xi_i(B_{\xi_1} \cap  \mathbb{R}^N_-)  \\
\text{and}\  d_e(\Xi_i(X))= (X_N + \lambda^{1/\alpha})_+ - \lambda^{1/\alpha}, \quad \forall \ X \in \Xi_{i}^{-1}(K_i) \subset B_{\xi_0}
\end{array}
\end{equation}
for some $\eta_1 < \eta^*$ and $\eta_2 >0$ such that $0< \xi_1 < \frac{\lambda^{1/\alpha}}{2}.$ 
Set $\widehat{K}_{i}= K_i \cap K.$ Then, splitting {$K\times K=\mathcal Q \cap (K\times K\setminus \mathcal Q)$ where 
$$\mathcal Q=\bigcup_{i\in I}\left(\widehat K_i\times \bigcup_{j\notin J^i}\widehat K_j\right)\cup\bigcup_{i\in I}\left(\widehat K_i\times \bigcup_{j\in J^i}\widehat K_j\cap (B_{R_i}(x_i))^c\right),$$
we obtain
from \eqref{def:adjacent}- \eqref{est:newcover}
\begin{equation}\label{est:integral1}
\iint_{\mathcal{Q}} \frac{|\underline{w}_\rho(x)-\underline{w}_\rho(y)|^{p}}{|x-y|^{N+sp}} ~dx~dy =\iint_{\mathcal{Q}} \frac{|(d(x) + \lambda^{\frac{1}{\alpha}})^\alpha - (d(y) + \lambda^{\frac{1}{\alpha}})^\alpha|^{p}}{|x-y|^{N+sp}} ~dx~dy \leq C_{\Omega, \eta}
\end{equation}
and for the second part, we perform change of variables using \eqref{diff:newpro} and diffeomorphisms $\Xi_i$ 
\begin{equation}\label{est:integral2}
\begin{split}
&\iint_{K\times K\setminus \mathcal Q}  \frac{|\underline{w}_\rho(x)-\underline{w}_\rho(y)|^{p}}{|x-y|^{N+sp}} ~dx~dy \\
=& \iint_{\Xi_i^{-1}(\widehat{K}_i)\times\Xi_i^{-1}(\widehat{K}_i)}   \frac{|(d(\Phi(X)) + \lambda^{\frac{1}{\alpha}})^\alpha - (d(\Phi(Y)) + \lambda^{\frac{1}{\alpha}})^\alpha|^{p}}{|\Phi_i(X)-\Phi_i(Y)|^{1+sp}} J_{\Xi_i}(X) J_{\Xi_i}(Y) ~dX~dY\\
& + \sum_{i\in I}\sum_{j \in J^i} \int_{\Xi_i^{-1}(\widehat{K}_i)}  \int_{\Xi_i^{-1}(\widehat{K}_{j} \cap B_{R_{i}}(x_{i}))}  \frac{|(d(\Phi(X)) + \lambda^{\frac{1}{\alpha}})^\alpha - (d(\Phi(Y)) + \lambda^{\frac{1}{\alpha}})^\alpha|^{p}}{|\Phi_i(X)-\Phi_i(Y)|^{1+sp}} J_{\Xi_i}(X) J_{\Xi_i}(Y) ~dX~dY\\
\leq&  C_{\Phi_i} \bigg( \iint_{\Xi_i^{-1}(\widehat{K}_i) \times\Xi_i^{-1}(\widehat{K}_i)} \frac{|(X_N + \lambda^{\frac{1}{\alpha}})_+^\alpha-(Y_N + \lambda^{\frac{1}{\alpha}})_+^\alpha|^{p}}{|X_N-Y_N|^{N+sp}} ~dX~dY \\ 
&+ \sum_{i\in I}\sum_{j \in J^i}\int_{\Xi_i^{-1}(\widehat{K}_i)}  \int_{\Xi_i^{-1}(\widehat{K}_{j} \cap B_{R_{i}}(x_{i}))}  \frac{|(X_N + \lambda^{\frac{1}{\alpha}})_+^\alpha-(Y_N + \lambda^{\frac{1}{\alpha}})_+^\alpha|^{p}}{|X_N-Y_N|^{N+sp}} ~dX~dY\bigg).
\end{split}
\end{equation}
Hence by observing  that $X_N, Y_N >-\min\{\xi_0,\xi_1\}> - \frac{\lambda^{1/\alpha}}{2}$} for all $X,Y \in \Xi_i^{-1}(\widehat{K}_i)$ and by using the same argument as in Theorem \ref{thest} and by combining \eqref{est:integral1} and \eqref{est:integral2}, we obtain $\underline w_\rho \in \overline{W}^{s,p}(\Omega_\eta).$ Similarly, we can prove $\overline w_\rho \in \overline{W}^{s,p}(\Omega_\eta).$
\end{proof}
\section{Sobolev and H\"older regularity}\label{Sob-Hol}
\noindent Introduce the following exponent
$$\alpha^\star_0:= \frac{sp-\delta}{p-1}.$$
We consider the sequence of function $\{\tilde K_{\lambda,\delta}\}_{\lambda \geq 0}$ where {$\delta \in (0,sp)$}, $\tilde K_{\lambda,\delta}: \mathbb{R}^N \to \mathbb{R}_+$  such that 
\begin{equation*}
 \begin{aligned}
\tilde K_{\lambda,\delta}(x)=  \left\{
 \begin{array}{ll}
 (K_\delta^{-\frac{1}{\delta}}(x)+ \lambda^{\frac{1}{\alpha^\star_0}})^{-\delta}  &  \text{ if } x \in \Omega, \\
 0 & \text{ if } x \notin \Omega,\\
 \end{array} 
 \right.
 \end{aligned}
 \end{equation*}
satisfying $\tilde K_{\lambda,\delta} \nearrow K_\delta$ a.e. in $\Omega$ as {$\lambda \to 0^+$}, and there exist two positive constants $\mathcal{D}_3, \mathcal{D}_4$ such that 
\begin{equation}\label{qasda}
\frac{\mathcal{D}_3}{\left(d(x)+ \lambda^{\frac{1}{\alpha^\star_0}} \right)^\delta} \leq \tilde K_{\lambda, \delta}(x) \leq \frac{\mathcal{D}_4}{\left(d(x)+ \lambda^{\frac{1}{\alpha^\star_0}} \right)^\delta}.
\end{equation}
Gathering Proposition \ref{prelim}, Theorem \ref{esix-prob} and Remark \ref{rk31}, we have the following result for the following approximated problem (noting $\gamma=0$ in Proposition \ref{prelim}):
\begin{equation*}\label{weightapp}
    (S_{\lambda}^\delta) \left\{
         \begin{alignedat}{2} 
             {} (-\Delta)^s_p u
             & {}=\tilde K_{\lambda,\delta}
             && \quad\mbox{ in } \, \Omega;   
             \\
             u & {}= 0
             && \quad\mbox{ in }\, \mathbb{R}^N \setminus \Omega.   
          \end{alignedat}
     \right.
\end{equation*}
\begin{thm}\label{weithm}
Let  $\Omega$ be a bounded domain with Lipschitz boundary.
Then there exists a increasing sequence of weak solution $\{u_\lambda\}_{\lambda>0} \subset W_0^{s,p}(\Omega) \cap L^\infty(\Omega)$ of $(S^\delta_{\lambda})$ such that 
\begin{equation}\label{ews}
\iint_{\mathbb{R}^{2N}} \frac{[u_\lambda(x)-u_\lambda(y)]^{p-1} (\phi(x)-\phi(y))}{|x-y|^{N+sp}} ~dx ~dy = \int_{\Omega} \tilde K_{\lambda,\delta}(x) \phi ~dx.
\end{equation}
 for all $\phi \in W_0^{s,p}(\Omega)$ and a minimal weak solution $u$ of $(S_{0}^\delta)$ such that  $u_\lambda^{\theta_1} \to u^{\theta_1} $ in $W_0^{s,p}(\Omega)$ and
 \begin{equation*}
\iint_{\mathbb{R}^{2N}} \frac{[u(x)-u(y)]^{p-1} (\varphi(x)-\varphi(y))}{|x-y|^{N+sp}} ~dx ~dy = \int_{\Omega} K_\delta(x) \varphi  ~dx
\end{equation*}
for all $\varphi \in \mathbb T $ where 
$\begin{aligned}
 \theta_1=  \left\{
 \begin{array}{ll}
 1 &  \text{ if } 0< \delta < 1+ s- \frac{1}{p}, \\
 \theta_2 & \text{otherwise} ,\\
 \end{array} 
\right. \ \text{and} \ \ \theta_2 > \max\{ \frac{sp-1}{p\alpha^\star_0}, 1\}.  
\end{aligned}
$
\end{thm} 
\noindent Let $\lambda_{s,p}$ be the first eigenvalue and $\varphi_{s,p}$ be a positive eigenfunction for the operator $(-\Delta)^s_p$. There exists a constant $c>0$ such that $ \frac1c d^s(x) \leq \varphi_{s,p}(x) \leq cd^s(x)$ for any $x \in \Omega$. {Indeed, the upper estimate can be retrieved Theorem 3.2 in \cite{FP} and Theorem 4.4 in \cite{IMS}, and the lower estimate from Theorem 1.1 in \cite{IMS} and Theorem 1.5 in \cite{Pe-quaas}}. Hence, from \eqref{qasda}, for any $\delta>0$, choosing a constant $a>0$ small enough, the following inequality holds for any $x\in \Omega$ and $\lambda\geq 0$:
\begin{equation*}
(-\Delta)^s_p (a \varphi_{s,p})
\leq \tilde K_{\lambda,\delta}(x) \leq (-\Delta)^s_p u_{\lambda}.
\end{equation*}
Then, by using Proposition $2.10$ in \cite{IMS}, we get, for any $\delta \in (0,sp)$, there exists a constant $\kappa_1$ such that for any $\lambda\geq 0$
\begin{equation}\label{22}
\kappa_1 d^{s}(x) \leq u_{\lambda}(x) \ \text{for any} \ x \in \Omega.
\end{equation}
Moreover, we have the {upper bound of $u_\lambda$ in $\Omega \setminus \Omega_\eta$. For $\eta>0$ small enough}, we consider $\{{B_{\frac{\eta}{4}}}(x_i)\}_{i \in \{1,2,\dots m\} }$  a finite covering of $ \overline{\Omega \setminus \Omega_{\eta}}$  such that
\begin{equation}\label{cover:omega:new}
  \overline{\Omega \setminus \Omega_{\eta}} \subset  \bigcup_{i=1}^m  {B_{\frac{\eta}{4}}(x_i)} \subset \Omega\setminus \Omega_{\frac\eta2}.  
\end{equation}
Then, from Theorem 3.2 and Remark 3.3 in \cite{BLS}, we deduce for any $i \in \{1,2,\dots,m\}$ 
\begin{equation}\label{est:loc-bdd}
\begin{split}
    \|u_\lambda\|_{L^\infty({ B_{\frac{\eta}{4}}(x_i)})} \leq& C \bigg[ \left( \fint_{{ B_{\frac{\eta}{2}}(x_i)}} |u_\lambda(x)|^p~dx\right)^{\frac{1}{p}} + \left(\eta^{sp}\int_{\R^N\setminus B_{\frac{\eta}{4}}(x_i)} \frac{|u_\lambda(x)|^{p-1}}{|x-x_i|^{N+sp}}~dx\right)^{\frac{1}{p-1}}  \\
     &+ \left(\eta^{sp}  \|\tilde{K}_{\lambda, \delta}\|_{L^\infty({ B_{\frac{\eta}{2}}(x_i)})}\right)^{\frac{1}{p-1}}\bigg]\
    \end{split}
\end{equation}
where $C$ depends upon $N,\,p$ and $s$. From the proof of Theorem \ref{esix-prob},  $\{u_\lambda^{\theta_{1}}\}_\lambda$ is uniformly bounded in $W_0^{s,p}(\Omega)$ and Sobolev embedding implies 
\begin{equation}\label{est:loc-bdd:exp1}
     \left(\fint_{{ B_{\frac{\eta}{2}}(x_i)}} |u_\lambda|^p\right)^{\frac{1}{p}} \leq c ( 1 + \|u_\lambda^{\theta_1}\|_{L^p(\Omega)}) \leq c (1+ \|u_\lambda^{\theta_1}\|_{s,p})\leq c.
\end{equation}
In the same way, the second term of the right hand-side is controlled, up to a constant independent of $\lambda$, by
\begin{equation}\label{est:loc-bdd:exp2}
        \left(\eta^{sp} \int_{\Omega\setminus B_{\frac{\eta}{4}}(x_i)} \frac{|u_\lambda(x)|^{p-1}}{\eta^{N+sp}} ~dx\right)^\frac{1}{p-1}  \leq  \eta^{-\frac{N}{p-1}} \|u_\lambda
        \|_{L^{p-1}(\Omega)} \leq c.
    \end{equation}
For the last term, for any $x\in \Omega\setminus \Omega_{\frac\eta2}$, we have
 \begin{equation*}
     |\tilde K_{\lambda, \delta}(x)| \leq \frac{\mathcal{D}_4}{\left(d(x)+ \lambda^{\frac{1}{\alpha^\star_0}} \right)^\delta} \leq c {\eta}^{-\delta} \leq c.
 \end{equation*}
Each constant in the previous estimates are independent of $\lambda$ {but depends on $\eta$}. Finally, plugging the three previous estimates into \eqref{est:loc-bdd} we deduce that for any $\eta>0$, there exists $\kappa_\eta>0$  independent of $\lambda$ such that
\begin{equation}\label{kapeta}\|u_{\lambda}\|_{L^\infty(\Omega \setminus \Omega_{\eta})}\leq \kappa_\eta.
\end{equation}
\noindent Now, we prove the sharp estimates for both upper and lower boundary behavior of the minimal weak solution for problem $(S_0^\delta)$ for different range of $\delta$. In this regard, we prove the following results with the help of comparison principle:

\begin{thm}\label{bdybeh}
Let $\Omega$ be a bounded domain with $C^{1,1}$ boundary and $u$ be minimal weak solution of the problem $(S_0^\delta)$. Then, we have
\begin{enumerate} 
\item[(i)] For $\delta\in (s,sp)$, there exists a positive constant $\Upsilon_1$ such that for any $x\in \Omega$, 
\begin{equation*}\label{large}
\frac1\Upsilon_1 d^{\alpha^\star_0}(x) \leq u(x) \leq \Upsilon_1 d^{\alpha^\star_0}(x).
\end{equation*}
\item[(ii)] For $\delta\in (0,s)$, for any $\epsilon>0$, there exist positive constants $\Upsilon_2$ and $\Upsilon_3=\Upsilon_3(\epsilon)$  such that for any $x \in \Omega$:
\begin{equation*}\label{equal}
\Upsilon_2 d^{s}(x) \leq u(x) \leq \Upsilon_3  d^{s-\epsilon}(x).
\end{equation*}
\end{enumerate}
\end{thm}
\begin{proof} Let $u_\lambda$ be the solution of $(S^\delta_\lambda)$ {for $\lambda < \lambda^*$, $\eta < \eta^*$} and $\rho>0$ given by Theorem \ref{thq}.\\ 
We begin to prove {\it (i)}. Take $\alpha=\frac{sp-\delta}{p-1}=\alpha^\star_0<s$ implying $sp-\alpha(p-1)=\delta$ and we define, for some $\eta >0$, $${\underline{u}^{(\lambda)}}=\min\{ \kappa_2(\frac{\eta}{2})^{s-\alpha} , \left(\frac{\mathcal{D}_3}{C_6}\right)^{\frac{1}{p-1}}\}\ \underline w_\rho= \underline c_\eta\underline w_\rho$$
and 
$${\overline{u}^{(\lambda)}}= \max\{(\frac{2}{\eta})^\alpha \kappa_{\frac\eta2}, \left(\frac{\mathcal{D}_4}{C_5}\right)^{\frac{1}{(p-1)}}\}\ \overline w_\rho= \overline c_\eta \overline w_\rho$$
where $\overline{w}_\rho$ and $\underline w_\rho$  satisfies \eqref{mainest}, $0<\kappa_2<\kappa_1$, $C_5$, $C_6$ are defined in \eqref{mainest}, $\kappa_1$ and $\kappa_{\frac\eta2}$ are defined in \eqref{22} and \eqref{kapeta}  respectively and $\mathcal D_3, \mathcal D_4$ are defined in \eqref{qasda}. Note $\underline c_\eta$ and $\overline c_\eta$ are independent of $\lambda$. \\
Hence for any $\lambda>0$, $u_\lambda$ satisfies
\begin{equation}\label{cons}
{\underline{u}^{(\lambda)}}(x) \leq 
u_{\lambda}(x)\leq {\overline{u}^{(\lambda)}}(x)\  \text{for}\ x \in \Omega \setminus \Omega_{\frac{\eta}{2}},\   \text{ and }\ {\underline{u}^{(\lambda)}}(x) \leq 0=u_{\lambda}(x)={\overline{u}^{(\lambda)}}(x) \ \text{for} \ x \in \Omega^c.
\end{equation}
Precisely, from \eqref{22}, \eqref{kapeta} and the definitions of $\underline{w}_\rho,\, \overline{w}_\rho$ given by \eqref{sub} and \eqref{super}, we get for $x \in \Omega \setminus \Omega_{\frac{\eta}{2}}$
\begin{equation*}
\begin{split}
    \underline{u}^{(\lambda)}= \underline{c}_\eta \underline{w}_\rho &\leq \kappa_2 (\frac{\eta}{2})^{s-\alpha} \underline{w}_\rho \leq \kappa_1 (\frac{\eta}{2})^{s-\alpha} d^\alpha(x) \leq \kappa_1 d^{s}(x) \leq u_\lambda(x),\\
    & u_\lambda(x) \leq \kappa_{\frac{\eta}{2}} \leq \kappa_{\frac{\eta}{2}} (\frac{2}{\eta})^\alpha d^\alpha(x) \leq \overline{c}_\eta \overline{w}_\rho = \overline{u}^{(\lambda)}.
    \end{split}
\end{equation*}
Moreover, from \eqref{mainest} and \eqref{qasda} and the choice of constants, we get $(-\Delta)^s_p {\underline{u}^{(\lambda)}} \leq(-\Delta)^s_p u_{\lambda} \leq (-\Delta)^s_p {\overline{u}^{(\lambda)}}$ weakly in $\Omega_{\frac{\eta}{2}}$ {\it i.e.} 
for any $\phi \in W_0^{s,p}(\Omega_{\eta})$, $\phi\geq0$:
\begin{equation*}
\begin{split}
\iint_{\mathbb{R}^{2N}} \frac{[{\underline{u}^{(\lambda)}(x)-\underline{u}^{(\lambda)}}(y)]^{p-1}(\phi(x)-\phi(y))}{|x-y|^{N+sp}} ~dx ~dy&\leq \iint_{\mathbb{R}^{2N}} \frac{[u_{\lambda}(x)-u_{\lambda}(y)]^{p-1}(\phi(x)-\phi(y))}{|x-y|^{N+sp}} ~dx ~dy\\
&\leq \iint_{\mathbb{R}^{2N}} \frac{[{\overline{u}^{(\lambda)}(x)-\overline{u}^{(\lambda)}}(y)]^{p-1}(\phi(x)-\phi(y))}{|x-y|^{N+sp}} ~dx ~dy.
\end{split}
\end{equation*}
Since $\underline w_\rho,\ \overline w_\rho\in \overline W^{s,p}(\Omega_\eta)$ and $u_\lambda \in  W^{s,p}_0(\Omega) \cap L^\infty(\Omega) \subset \overline W^{s,p}(\Omega_\eta)$, Proposition 2.10 in \cite{IMS}  in $\Omega_{\eta}$ implies {$\underline{u}^{(\lambda)} \leq u_{\lambda}\leq \overline{u}^{(\lambda)}$} in $\Omega_\eta$. Hence, from \eqref{cons} and passing $\lambda \to 0$, we deduce  {\it (i)}.
\vspace{0.15cm}\\
Now we prove {\it (ii)} {\it i.e.} the case $\delta \leq s$. Since \eqref{22} holds, it remains to obtain the upper bound estimate.\\ Let $\tilde{u}_{\lambda} \in W_0^{s,p}(\Omega)$ be the weak solution of $(S^{\tilde\delta}_{ \lambda})$ with $\tilde\delta= s+\epsilon(p-1)>s$ and for $\epsilon >0$. Then, choosing a suitable constant $c_\epsilon>0$ independent of $\lambda$, ${\tilde u^{(\lambda)}}=c_\epsilon\tilde u_\lambda$ is a weak supersolution of $(S^{\delta}_{ \lambda})$.  
Hence by Proposition $2.10$ in \cite{IMS}, we have $u_\lambda\leq {\tilde u^{(\lambda)}}$ in $\Omega$. We pass to the limit as $\lambda \to 0$ and using {\it (i)} with $\tilde u(x)=\lim_{\lambda\to 0} {\tilde u^{(\lambda)}}(x)$, we get, for $\epsilon>0$,
$u(x) \leq \overline{c}_{\eta,\epsilon} d^{s-\epsilon}(x)$ for $x\in\Omega$.
\end{proof}
\noindent Concerning the H\"older regularity of the weak solution of the problem $(P)$, we prove Theorem \ref{thm:regularity}:  \\
\textbf{Proof of Theorem \ref{thm:regularity}} 
Let $u$ be the minimal solution of the problem $(P).$ First, we prove the boundary behavior of the minimal weak solution by dividing the proof into two cases:  \\
\textbf{ Case 1:} $\delta-s(1-\gamma)\leq 0$  \\
Let $\tilde{u}$ and $\dot{u}$ are weak solution of the problem $(S_0^\zeta)$ for {$\zeta= \delta+ \gamma s \leq s$} and $\zeta= \delta+ \gamma (s-\epsilon)<s$ respectively for ${\epsilon\in (0,s)}.$ Then, from Theorem \ref{bdybeh} $(ii)$ there exist constants $c_i>0$ such that 
$$c_1 d^{s}(x) \leq \tilde{u}(x) \leq c_2 d^{s-\epsilon}(x), \ c_3 d^{s}(x) \leq \dot{u}(x) \leq c_4 d^{s-\epsilon}(x)\ \text{in} \ \Omega$$ 
and $\dot{u}, \tilde{u}$ satisfies
{
 $$(-\Delta)^s_p \left(C_* \dot{u}\right)=  \frac{\mathcal{C}_1}{\mathcal{C}_2 c_4^{\gamma}} K_{\delta+\gamma(s-\epsilon)}(x) \leq  \frac{\mathcal{C}_1}{d^{\delta}(x)(c_2 d^{s-\epsilon}(x))^\gamma} \leq\frac{\mathcal{C}_1}{d^{\delta}(x) \dot{u}^\gamma} \leq \frac{K_\delta(x)}{\dot{u}^\gamma}$$
 and $$\frac{K_\delta(x)}{\tilde{u}^\gamma} \leq \frac{\mathcal{C}_2}{d^{\delta}(x) \tilde{u}^\gamma} \leq \frac{\mathcal{C}_2}{d^{\delta}(x) (c_1 d^s(x))^\gamma} \leq  \frac{\mathcal{C}_2}{\mathcal{C}_1 c_1^{\gamma}} K_{\delta+ \gamma s}(x) = (-\Delta)^s_p \left(C^* \tilde{u}\right)$$
where $C_*= \left(\frac{\mathcal{C}_1}{\mathcal{C}_2 c_4^\gamma}\right)^\frac{1}{p-1}$ and $C^*= \left(\frac{\mathcal{C}_2}{\mathcal{C}_1  c_1^\gamma}\right)^\frac{1}{p-1}$} and $\mathcal{C}_1, \mathcal{C}_2$ are defined in \eqref{eq:sing-weight}. Then by applying Theorem \ref{WCP}, we get
\begin{equation}\label{est:boundbeh1}
    C_1 d^{s}(x) \leq u(x) \leq C_2 d^{s-\epsilon}(x)\ \text{in} \ \Omega
\end{equation} 
for every $\epsilon>0$, $C_1= c_1 C_*$ and $C_2= c_4 C^*$.\vspace{0.2cm} \\
\textbf{ Case 2:} $\delta-s(1-\gamma)>0$  \\
Let $\lambda>0$ and $u_{\lambda} \in W_0^{s,p}(\Omega)$ be the solution of the problem $(P_{\lambda}^\gamma)$ {for $\lambda < \lambda^*$ given in Theorem \ref{thq}.} \\
By considering the same cover of $\overline{\Omega \setminus \Omega_{\eta}}$ as in \eqref{cover:omega:new} and applying Theorem 3.2 and Remark 3.3 in \cite{BLS}, we obtain, 
\begin{equation}\label{est:sing:loc-bdd}
\begin{split}
    \|u_\lambda\|_{L^\infty( B_{\frac{\eta}{4}}(x_i))} \leq & \ C \bigg[ \left( \fint_{B_{\frac{\eta}{2}}(x_i)} |u_\lambda(x)|^p~dx\right)^{\frac{1}{p}} + \left(\eta^{sp}\int_{\R^N\setminus B_{\frac{\eta}{4}}(x_i)} \frac{|u_\lambda(x)|^{p-1}}{|x-x_i|^{N+sp}}~dx\right)^{\frac{1}{p-1}}  \\
     &+ \left(\eta^{sp}  \left\|\frac{K_{\lambda, \delta}}{(u_\lambda+\lambda)^\gamma}\right\|_{L^\infty(B_{\frac{\eta}{2}}(x_i))}\right)^{\frac{1}{p-1}}\bigg]\
    \end{split}
\end{equation}
for any $i \in \{1,2,\dots,m\}$ where $C$ depends upon $N,\,p$ and $s$. By repeating the same arguments as in \eqref{22}, \eqref{est:loc-bdd:exp1} and \eqref{est:loc-bdd:exp2} we get that the first two terms in the right hand-side of \eqref{est:sing:loc-bdd} are bounded with bounds independent of $\lambda$ and
\begin{equation}\label{eigen:bdd}
    \varkappa d^s(x) \leq u_{\lambda}(x) \ \text{in} \ \Omega 
\end{equation}
for some $\varkappa>0$ independent of $\lambda.$ Now, by using above inequality, we estimate the last term in the right hand-side of \eqref{est:sing:loc-bdd}: for any $x\in \Omega\setminus \Omega_{\frac\eta2}$, we have
 \begin{equation*}
     \left|\frac{K_{\lambda, \delta}(x)}{(u_\lambda+\lambda)^\gamma}\right| \leq \frac{\mathcal{D}_4}{\left(d(x)+ \lambda^{\frac{1}{\alpha^\star_0}} \right)^{\delta} (\varkappa d^s(x))^\gamma} \leq c {\eta}^{-(\delta + \gamma s)} \leq c.
 \end{equation*}
Finally, we deduce that for any $\eta>0$, there exists $\varkappa_\eta>0$  independent of $\lambda$ such that
\begin{equation}\label{kapeta1}
\|u_{\lambda}\|_{L^\infty(\Omega \setminus \Omega_{\eta})}\leq \varkappa_\eta.
\end{equation}
 For $\alpha= \frac{sp-\delta}{p+\gamma-1}=\alpha^\star$ and $0< \eta <\eta^*$, define
 $$\underline u^{\lfloor \lambda \rfloor}= c_\eta\underline w_\rho \ \text{and } \overline u^{\lfloor \lambda \rfloor}= \dot{c}_\eta \overline w_\rho \ \text{such that} \  0< c_\eta \leq  \left(\frac{\eta}{2}\right)^{s-\alpha} \varkappa \ \text{and} \ \dot{c}_\eta \geq \left(\frac{2}{\eta}\right)^{\alpha} \varkappa_{\frac{\eta}{2}}$$
 where $\underline{w}_{\rho}$, $\overline w_\rho$,  $\varkappa$, $\varkappa_{\frac{\eta}{2}}$ and $\eta^*$ are defined in \eqref{sub}, \eqref{super}, \eqref{eigen:bdd}, \eqref{kapeta1} and Theorem \ref{thq} respectively.
We note that $c_\eta$, $\dot{c}_\eta$ are independent of $\lambda$ and for any $\lambda >0$, $\underline u^{\lfloor \lambda \rfloor}$ and $\overline u^{\lfloor \lambda \rfloor}$ satisfy
 \begin{equation}\label{cons0}
\underline u^{\lfloor \lambda \rfloor}(x)  \leq u_{\lambda}(x) \leq\overline u^{\lfloor \lambda \rfloor}(x) \ \text{for}\ x \in \Omega \setminus \Omega_{\frac{\eta}{2}}\  \text{and}\ \underline u^{\lfloor \lambda \rfloor}(x) \leq u_{\lambda}(x) \leq\overline u^{\lfloor \lambda \rfloor}(x)\ \text{for} \ x \in \Omega^c.
\end{equation}
Using the definition of $\underline{w}_\rho$ and $\overline{w}_\rho$ in \eqref{sub} and \eqref{super} respectively and estimates in \eqref{mainest}, we obtain $$(\underline u^{\lfloor \lambda \rfloor} + \lambda)=c_\eta (d(x)+ \lambda^{1/\alpha})^{\alpha} + \lambda \left(1-c_\eta \right) \ \text{and} \ (\overline u^{\lfloor \lambda \rfloor} + \lambda)= \dot{c}_\eta (d(x)+ \lambda^{\frac{1}{\alpha}})^{\alpha} + \lambda  \ \text{in} \ \Omega$$ and 
\begin{equation}\label{eq:sub}
    (-\Delta)^s_p \underline u^{\lfloor \lambda \rfloor} \leq \frac{ c_\eta^{p-1} C_6}{(d(x)+ \lambda^{\frac{1}{\alpha}})^{\delta+\alpha \gamma}} \leq c_\eta^{p-1} \frac{C_6 K_{\lambda,\delta}(x)}{\mathcal{C}_3(d(x)+ \lambda^{\frac{1}{\alpha}})^{\alpha \gamma}}\quad \text{weakly in}\ \Omega_\eta,
\end{equation}
\begin{equation}\label{eq:super}
 (-\Delta)^s_p\overline u^{\lfloor \lambda \rfloor} \geq \frac{ C_5 \dot{c}_\eta ^{p-1}}{(d(x)+ \lambda^{\frac{1}{\alpha}})^{\delta+\alpha \gamma}} \geq \dot{c}_\eta^{p-1} \frac{ C_5 K_{\lambda,\delta}(x)}{\mathcal{C}_4 (d(x)+ \lambda^{\frac{1}{\alpha}})^{\alpha \gamma}} \quad \text{weakly in}\ \Omega_\eta  
\end{equation}
where $C_5, C_6$ and $\mathcal{C}_3$, $\mathcal{C}_4$ are defined in \eqref{mainest} and \eqref{qasd} respectively. Since $\dot{c}_\eta \to \infty$ as $\eta \to 0$ and $(\overline u^{\lfloor \lambda \rfloor} + \lambda) \geq \dot{c}_\eta  (d(x)+ \lambda^{\frac{1}{\alpha}})^{\alpha}$, we can choose $\eta$ small enough (independent of $\lambda$)  such that $ C_5 \dot{c}_\eta^{\gamma+ p-1} \geq \mathcal{C}_4$ and \eqref{eq:super} reduced to
$$(-\Delta)^s_p\overline u^{\lfloor \lambda \rfloor} \geq  \frac{K_{\lambda,\delta}(x)}{(\overline u^{\lfloor \lambda \rfloor} + \lambda)^\gamma}\ \text{weakly in }  \Omega_{\eta}.$$
Now to prove similar estimate for $\underline u^{\lfloor \lambda \rfloor}$, we divide the proof into two cases; for $x\in \Omega_\eta$:\\
\textbf{Case (i)}: $c_\eta (d(x)+ \lambda^{1/\alpha})^{\alpha} \geq  \lambda \left(1-c_\eta \right)$\\
In this case, we have $(\underline u^{\lfloor \lambda \rfloor} + \lambda)^{-\gamma} \geq (2 c_\eta)^{-\gamma}  (d(x)+ \lambda^{1/\alpha})^{-\alpha\gamma}$ and by choosing $\eta$ small enough such that $2^\gamma c_\eta^{\gamma+p-1} \leq \frac{ \mathcal{C}_3}{C_6}$, \eqref{eq:sub} reduced to 
 $$(-\Delta)^s_p \underline u^{\lfloor \lambda \rfloor} 
 \leq \frac{2^\gamma c_\eta^{\gamma+p-1} C_6}{\mathcal{C}_3} \frac{K_{\lambda,\delta}(x)}{(\underline u^{\lfloor \lambda \rfloor} + \lambda)^\gamma} \leq  \frac{K_{\lambda,\delta}(x)}{(\underline u^{\lfloor \lambda \rfloor} + \lambda)^\gamma}.$$
\textbf{Case (ii)}: $c_\eta  (d(x)+ \lambda^{1/\alpha})^{\alpha} \leq  \lambda \left(1-c_\eta \right)$\\
In this case, we have  $(\underline u^{\lfloor \lambda \rfloor} + \lambda)^{-\gamma} \geq (2\lambda)^{-\gamma} \left(1-c_\eta \right)^{-\gamma}$ and by choosing $\eta$ small enough such that $c_\eta \leq 1$ and $C_6 c_\eta^{p-1} \leq \mathcal{C}_3 {(2 \lambda^*)^{-\gamma}}(1-c_\eta )^{-\gamma}$, \eqref{eq:sub} reduced to,
{\[
(-\Delta)^s_p \underline u^{\lfloor \lambda \rfloor}
\leq \frac{c_\eta^{p-1} C_6 }{\mathcal{C}_3} \frac{K_{\lambda,\delta}(x)}{(\underline u^{\lfloor \lambda \rfloor}+ \lambda)^\gamma} (2\lambda)^\gamma \left(1-c_\eta \right)^{\gamma} \leq \frac{K_{\lambda,\delta}(x)}{(\underline u^{\lfloor \lambda \rfloor} + \lambda)^\gamma}. \]}
Therefore, in each case, we can choose $\eta$ small enough (independent of $\lambda$) such that
$$(-\Delta)^s_p \underline u^{\lfloor \lambda \rfloor} \leq  \frac{K_{\lambda,\delta}(x)}{(\underline u^{\lfloor \lambda \rfloor} + \lambda)^\gamma}\ \text{weakly in }  \Omega_\eta.$$
Since $\underline u^{\lfloor \lambda \rfloor},\overline u^{\lfloor \lambda \rfloor} \in \overline W^{s,p}(\Omega_\eta)$ and $u_\lambda \in L^\infty(\Omega) \cap W^{s,p}_0(\Omega) \subset \overline W^{s,p}(\Omega_\eta)$, Proposition 2.10 in \cite{IMS}  in $\Omega_{\eta}$ implies $\underline u^{\lfloor \lambda \rfloor} \leq u_{\lambda}\leq\overline u^{\lfloor \lambda \rfloor}$ in $\Omega_\eta$. Hence, from \eqref{cons0} and passing $\lambda \to 0$,
\begin{equation}\label{est:boundbeh2}
C_1 d^{\alpha^\star} \leq u  \leq C_2 d^{\alpha^\star}\ \text{in} \ \Omega.
\end{equation}
where $C_1= c_\eta$ and $C_2= \dot{c}_\eta.$
 \\
\noindent \textbf{Interior and boundary regularity:} First we claim the following:  \\
\textbf{Claim:} For all $x_0 \in \Omega \ \text{and} \ R_0=\frac{d(x_0)}{2}$ there exists universally $C_\Omega >0$, $0< \omega_1 <s$ and $0< \omega_2 \leq \alpha^\star$  such that
\begin{equation}\label{est:reg1}
  \text{if}\ 1 < p < 2,\ \left\{
  \begin{array}{l l}
  \|u\|_{C^{\omega_1}(B_{R_0}(x_0))} \leq C_{\Omega} \ &\text{for} \ \delta\leq s(1-\gamma), \\
  \ &\\
  \|u\|_{C^{\omega_2}(B_{R_0}(x_0))} \leq C_{\Omega} \ &\text{for} \ \delta> s(1-\gamma)  
  \end{array}
  \right.
\end{equation}
and
\begin{equation}\label{est:reg2}
\text{if}\ 2 \leq p < \infty,\ \left\{
 \begin{array}{l l}
\|u\|_{C^{s-\epsilon}(B_{R_0}(x_0))} \leq C_{\Omega} \ &\text{for} \ \delta\leq s(1-\gamma), \\
& \\
\|u\|_{C^{\alpha^\star}(B_{R_0}(x_0))} \leq C_{\Omega} \ &\text{for} \ \delta> s(1-\gamma)  .
\end{array}
  \right.
\end{equation}
Let $x_0 \in \Omega$, $R_0= \frac{d(x_0)}{2}$ such that $B_{R_0}(x_0) \subset B_{2R_0}(x_0) \subset \Omega$ and $u \in W^{s,p}(B_{2R_0}(x_0)) \cap L^\infty(B_{2R_0}(x_0))$ be the minimal weak solution of $(P)$, then it satisfies  
$$(-\Delta)^s_p u = \frac{K_\delta(x)}{u^\gamma} \leq \frac{\mathcal{C}_2}{C_1^\gamma} \frac{1}{d^{\gamma s+\delta}} \ \leq \frac{\mathcal{C}_2}{C_1^\gamma} \frac{1}{R_0^{\gamma s+\delta}} \ \text{in} \ B_{R_0}(x_0)\ \text{for} \ \delta\leq s(1-\gamma)$$ and
$$(-\Delta)^s_p u = \frac{K_\delta(x)}{u^\gamma} \leq \frac{\mathcal{C}_2}{C_1^\gamma} \frac{1}{d^{\gamma\alpha^\star+\delta}}\  \leq \frac{\mathcal{C}_2}{C_1^\gamma} \frac{1}{R_0^{\gamma\alpha^\star+\delta}}\ \ \text{in} \ B_{R_0}(x_0)\  \text{for} \ \delta> s(1-\gamma)$$
where $\mathcal{C}_2$ is defined in \eqref{eq:sing-weight}.
{\noindent Then, by using Proposition \ref{intreg} for $p\in (1,2)$, \eqref{est:boundbeh1} and \eqref{est:boundbeh2} we obtain: there exist $\omega_1 \in(0,s)$ and $\omega_2\in(0, \alpha^\star]$} such that \\
if $\delta\leq s(1-\gamma):$
\begin{equation*}
\begin{split}
[u]_{C^{\omega_1}({B_{R_0}(x_0)}) }
\leq &CR_0^{-\omega_1} \left( R_0^{\frac{(sp-\delta-\gamma s)}{p-1}} + \|u\|_{L^\infty(B_{2R_0}(x_0))} + \left((2R_0)^{sp} \int_{(B_{2R_0}(x_0))^c} \frac{|u(y)|^{p-1}}{|x_0-y|^{N+sp}} ~dy\right)^\frac{1}{p-1} \right) \\
\leq & {\bf C_1}  
\end{split}
\end{equation*}
and if $\delta> s(1-\gamma)$:
\begin{equation*}
\begin{split}
[u]_{C^{\omega_2}({B_{R_0}(x_0)})} 
&\leq C R_0^{-\omega_2}\left( R_0^{\alpha^\star} + \|u\|_{L^\infty(B_{2R_0}(x_0))} + \left((2R_0)^{sp} \int_{(B_{2R_0}(x_0))^c} \frac{|u(y)|^{p-1}}{|x_0-y|^{N+sp}} ~dy\right)^\frac{1}{p-1} \right) \\
& \leq  {\bf C_2}. 
\end{split}
\end{equation*}
Furthermore, using Proposition \ref{pro:large-p} for $p\in[2,+\infty)$, we get for any $\epsilon>0$
\begin{equation*}
[u]_{C^{s-\epsilon}({B_{R_0/64}(x_0)})} \leq {\bf C_3} \ \ \text{if}\  \delta\leq s(1-\gamma)\quad \text{and}\quad  [u]_{C^{\frac{sp-\delta}{p+\gamma-1}}(B_{R_0/64}(x_0))} \leq {\bf C_4} \ \ \text{if}\ \delta> s(1-\gamma).
\end{equation*} 
The constants ${\bf C_i}$ are independent of the choice of point $x_0$ (and $R_0$) and since $u\in L^\infty(\Omega)$ we deduce \eqref{est:reg1} and \eqref{est:reg2} and by a covering argument for any $\Omega' \Subset \Omega$, we conclude
\begin{equation}\label{est:reg3}
  \text{if}\ 1 < p < 2,\ \left\{
  \begin{array}{l l}
  \|u\|_{C^{\omega_1}(\Omega')} \leq C_{\Omega'} \ &\text{for} \ \delta\leq s(1-\gamma), \\
  \ &\\
  \|u\|_{C^{\omega_2}(\Omega')} \leq C_{\Omega'} \ &\text{for} \ \delta> s(1-\gamma)  
  \end{array}
  \right.
  \end{equation}
and
\begin{equation}\label{est:reg4}
\text{if}\ 2 \leq p < \infty,\ \left\{
 \begin{array}{l l}
\|u\|_{C^{s-\epsilon}(\Omega')} \leq C_{\Omega'} \ &\text{for} \ \delta\leq s(1-\gamma), \\
& \\
\|u\|_{C^{\alpha^\star}(\Omega')} \leq C_{\Omega'} \ &\text{for} \ \delta> s(1-\gamma)  .
\end{array}
  \right.
\end{equation}
Now, to prove the regularity estimate in $\Omega$ (and then the whole $\mathbb{R}^N$) since $u=0$ in $\mathbb{R}^N \setminus \Omega$, it is {sufficient from interior regularity that follows from \eqref{est:reg3}, \eqref{est:reg4}}, to prove  \eqref{est:reg3} and \eqref{est:reg4} on {
$\Omega_\eta$} where $\eta>0$ small enough.  \\
In this regard, let $x,y \in \Omega_\eta$ {and suppose without loss of generality $d(x) \geq d(y)$. Now two cases occur:  \\
\textbf{(I)} either $|x-y| \leq \frac{d(x)}{64}$, in which case set $64R_0= d(x)$ and $y\in B_{R_0}(x)$. Hence we apply \eqref{est:reg1} or \eqref{est:reg2} in $B_{R_0}(x)$ and  we obtain the regularity.\\
\textbf{(II)} or $|x-y| \geq \frac{d(x)}{64}\geq  \frac{d(y)}{64}$ in which case \eqref{est:boundbeh1} and \eqref{est:boundbeh2} ensures for a constant $C>0$ }large enough, we get
\begin{equation}\label{est:bdry1}
\begin{split}
\frac{|u(x)-u(y)|}{|x-y|^{s-\epsilon}} &\leq \frac{|u(x)|}{|x-y|^{s-\epsilon}} + \frac{|u(y)|}{|x-y|^{s-\epsilon}} \leq C_1 \left( \frac{u(x)}{d^{s-\epsilon}(x)} + \frac{u(y)}{d^{s-\epsilon}(y)}\right) \leq C,
\end{split}
\end{equation}
and
\begin{equation}\label{est:bdry2}
\frac{|u(x)-u(y)|}{|x-y|^{\alpha^\star}} \leq \frac{|u(x)|}{|x-y|^{\alpha^\star}} + \frac{|u(y)|}{|x-y|^{\alpha^\star}} 
\leq C_2 \left( \frac{u(x)}{d^{\alpha^\star}(x)} + \frac{u(y)}{d^{\alpha^\star}(y)}\right) \leq C.
\end{equation}
Then, finally by combining \eqref{est:reg3}-\eqref{est:bdry2}, we get our claim and the proof is complete.
\qed \\
\textbf{Proof of Corollary \ref{cor:bdry+sobreg}}:  \\
For ${\delta > s(1-\gamma)}$, let $u_\epsilon$ be the weak solution of the problem $(P_\epsilon^\gamma).$ Then, using the boundary behavior of the approximating sequence $u_\epsilon$ and taking $\phi= u_\epsilon$ in \eqref{oiu}, we obtain 
$$\|u_\epsilon\|_{s,p}^p = \int_{\Omega} K_{\epsilon,\delta}(x) u_\epsilon^{1-\gamma} ~dx \leq C_1 \int_{\Omega} d^{(1-\gamma)\alpha^\star -\delta}(x) ~dx \leq C$$
{if $ (1-\gamma)(sp-\delta) >(\delta-1)(p+\gamma -1) \Leftrightarrow sp(\gamma-1) +\delta p < (p+\gamma-1) \Leftrightarrow \Lambda<1.$} \\
Similarly, by taking $\phi= u^\theta_\epsilon$ in \eqref{oiu} and {using Proposition \ref{prelim2}}, we obtain for $\theta> \Lambda>1$
\begin{equation*}
        \|u^\theta_\epsilon\|_{s,p}^p \leq  \theta^{p-1} \int_{\Omega} K_{\epsilon,\delta}(x) u_\epsilon^{(\theta-1)(p-1)+ \theta-\gamma} ~dx \leq C_2 \int_{\Omega} d^{(\theta p-(p-1+\gamma))\alpha^\star -\delta}(x) ~dx\leq C.
\end{equation*}
Now, by passing limits $\epsilon \to 0$ in \eqref{oiu}, we get the minimal solution $u \in W_0^{s,p}(\Omega)$ if $\Lambda <1$ and $u^\theta\in W_0^{s,p}(\Omega)$ if $ \theta> \Lambda>1$.\\
The only if statement follows from the Hardy inequality and the boundary behavior of the weak solution. Precisely, if $ \Lambda \geq 1$, then $u \notin W_0^{s,p}(\Omega)$. Indeed, we have
$$ \|u\|_{s,p}^p\geq C\int_{\Omega} \left|\frac{u(x)}{d^{s}(x)}\right|^p ~dx \geq C\int_{\Omega} d^{p(\alpha^\star-s)}(x) ~dx =\infty.$$
In the same way, if $ \theta \in[1,\Lambda]$, then 
$$  \|u^\theta\|_{s,p}^p\geq C\int_{\Omega} \left|\frac{u^{\theta}(x)}{d^{s}(x)}\right|^p ~dx\geq C\int_{\Omega} d^{p(\theta\alpha^\star-s)}(x) ~dx =\infty$$
and we deduce $u^\theta \notin W_0^{s,p}(\Omega)$ .
\qed
{\begin{remark}
In case of local operator, {\it i.e.} p-Laplacian operator, the optimal condition of Sobolev regularity in Theorem $1.4$, \cite{deep} coincide with the our condition for $s=1$.
\end{remark}}
\section{Nonexistence result}\label{non-es}
\textbf{Proof of Theorem \ref{thm:non-exis}:} {$\delta \geq sp$}.  We proceed by contradiction assuming there exist a weak solution $u_0 \in W_{loc}^{s,p}(\Omega)$ of the problem $(P)$ and $\kappa_0 \geq 1$ such that $u_0^{\kappa_0} \in W_0^{s,p}(\Omega)$.\\
We choose $\Gamma \in (0,1)$ and  $\delta_0 <sp$ such that $\Gamma K_{\delta_0}(x) \leq K_\delta(x)$ and the constant $\Gamma$ is independent of $\delta_0$ {for $\delta_0 \geq \delta_0^*$ with $\delta_0^*>0$.}\\
For $\epsilon>0$, let $u_\epsilon \in W_0^{s,p}(\Omega) \cap C^{0,\ell}(\overline{\Omega})$ be the unique weak solution of
\begin{equation}\label{bv}
\iint_{\mathbb{R}^{2N}} \frac{[u_\epsilon(x)- u_\epsilon(y)]^{p-1} (\phi(x)- \phi(y))}{|x-y|^{N+sp}} ~dx ~dy = \int_{\Omega} \frac{\Gamma K_{\epsilon, \delta_0}(x)}{(u_\epsilon+ \epsilon)^\gamma} \phi ~dx 
\end{equation}
for any $\phi \in W_0^{s,p}(\Omega)$.\\
By the continuity of $u_\epsilon$, for given $\theta >0$, there exists a $\eta {= \eta(\epsilon,\theta)} >0$ such that $u_\epsilon \leq  \frac{\theta}{2}$ in $\Omega_\eta$. Since $u_0 \geq 0$, then $ w:= u_\epsilon- u_0- \theta \leq - \frac{\theta}{2} <0$ in $\Omega_\eta$ and 
\begin{equation*}\label{supp}
\mbox{supp}({w}^+) \subset \mbox{supp}((u_\epsilon- \theta)^+) \subset \Omega \setminus \Omega_\eta.
\end{equation*}
{We have $w^+\in W_0^{s,p}(\tilde \Omega)\subset W_0^{s,p}(\Omega)$ for some $\tilde \Omega$ such that $\Omega\setminus \Omega_\eta \subset \tilde \Omega \Subset \Omega$. Hence, choosing $w^+$ as a test function in \eqref{bv}, we get
\begin{equation}\label{bv4}
\iint_{\mathbb{R}^{2N}} \frac{[u_\epsilon(x)- u_\epsilon(y)]^{p-1} (w^+(x)- w^+(y))}{|x-y|^{N+sp}} ~dx ~dy = \int_{\Omega} \frac{\Gamma K_{\epsilon, \delta_0}(x)}{(u_\epsilon+ \epsilon)^\gamma} w^+ ~dx \leq \int_{\Omega} \frac{\Gamma K_{\epsilon, \delta_0}(x)}{u_\epsilon^\gamma} w^+ ~dx. 
\end{equation} 
Moreover, $u_0$ is a weak solution of $(P)$ and taking $w^+\in W_0^{s,p}(\tilde \Omega)$ as test function in Definition \ref{def1} with $u_0$, we have
\begin{equation}\label{bv2}
\begin{split}
\iint_{\mathbb{R}^{2N}} \frac{[u_0(x)- u_0(y)]^{p-1} (w^+(x)- w^+(y))}{|x-y|^{N+sp}} ~dx ~dy = \int_{\Omega} \frac{K_{\delta}(x)}{u_0^\gamma} w^+ ~dx \geq \int_{\Omega} \frac{ \Gamma K_{\epsilon, \delta_0}(x)}{u_0^\gamma} w^+ ~dx. 
\end{split}
\end{equation}}
By subtracting \eqref{bv2} and \eqref{bv4}, we get
\begin{equation}\label{bv5}
\begin{split}
\iint_{\mathbb{R}^{2N}} &\frac{([u_\epsilon(x)- u_\epsilon(y)]^{p-1}-[u_0(x)- u_0(y)]^{p-1}) (w^+(x)- w^+(y))}{|x-y|^{N+sp}} ~dx ~dy\\
& \leq \int_{\Omega} \left(\frac{\Gamma K_{\epsilon,\delta_0}(x)}{u_\epsilon^\gamma} - \frac{\Gamma K_{\epsilon, \delta_0}(x)}{u_0^\gamma} \right) w^+ ~dx  \leq 0.
\end{split}
\end{equation}
{
Applying the following identity 
$$[b]^{p-1} - [a]^{p-1} = (p-1)(b-a) \int_0^1 |a+t(b-a)|^{p-2} ~dt $$
with $a=u_0(x)- u_0(y)$ and $b=u_\epsilon(x)- u_\epsilon(y)$, we get  
\begin{equation}\label{bv6}
        [u_\epsilon(x)- u_\epsilon(y)]^{p-1}-[u_0(x)- u_0(y)]^{p-1} 
         = (p-1) Q(x,y) (w(x)-w(y))
\end{equation}
where $$Q(x,y)= \int_0^1 |u_0(x)- u_0(y)+t(w(x)-w(y))|^{p-2} ~dt \geq 0.$$
Now by multiplying \eqref{bv6} with $(w^+(x)-w^+(y))$, we obtain
\begin{equation*}
     ([u_\epsilon(x)- u_\epsilon(y)]^{p-1}-[u_0(x)- u_0(y)]^{p-1}) (w^+(x)- w^+(y))=(p-1) Q(x,y)(w(x)-w(y)) (w^+(x)-w^+(y))\geq 0
\end{equation*}
since the mapping $x\to x^+$ is nondecreasing.\\
From \eqref{bv5}, we get $w^+= (u_\epsilon- u_0 -\theta)^+=0$ a.e. in $\Omega$}. Since  $\theta$ is arbitrary, we deduce $u_\epsilon \leq u_0$ in $\Omega.$  Using the estimates in {\bf Case 2} of the proof of Theorem \ref{thm:regularity}, we have
$$\eta c_\eta (d(x)+ \epsilon^{\frac{\gamma+p -1}{sp-\delta_0}})^{\frac{sp-\delta_0}{\gamma+p-1}} - \epsilon \leq u_\epsilon \leq u_0 \ \text{in} \ \Omega.$$ 
Now, by using Hardy inequality and $u_0^{\kappa_0} \in W_0^{s,p}(\Omega)$, we obtain
$$\displaystyle
(\eta c_\eta)^{\kappa_0 p}\int_{\Omega} \left|\frac{\left((d(x)+ \epsilon^{\frac{\gamma+p -1}{sp-\delta_0}})^{\frac{sp-\delta_0}{\gamma+p-1}} - \epsilon \right)^{\kappa_0 }}{d^{s}(x)}\right|^p ~dx \leq \int_{\Omega} \left|\frac{u_0^{\kappa_0 }}{d^{s}(x)}\right|^p ~dx < \infty.
$$
Now, by choosing $\delta_0$ close enough to $sp$ and by taking $\epsilon \to 0$, we obtain that the left hand side is not finite, which is a contradiction and hence claim. 
\appendix
\section{Appendix}\label{ann}
In this section, we recall the local regularity results for the p-fractional Laplacian operator. We set for $R>0$ and $y \in \mathbb{R}^N$
$$Q(u;y,R)= \|u\|_{L^\infty(B_R(y))} + 
\left(R^{sp} \int_{(B_R(y))^c} \frac{|u(x)|^{p-1}}{|x-y|^{N+sp}} ~dx\right)^\frac{1}{p-1}$$

\begin{pro}(Corollary 5.5, \cite{IMS})\label{intreg}
If $u \in \overline{W}^{s,p}(B_{2R_0}(y)) \cap L^\infty(B_{2R_0}(y))$ satisfies $|(-\Delta)^s_p u| \leq K$ weakly in $B_{2R_0}(y)$ for some $R_0>0$, then there exists universal constants $\omega\in (0,1)$ and $C>0$ with the following property:{
$$[u]_{C^{\omega}(B_{R_0}(x_0))}:=\sup_{x,y\in B_{R_0}(x_0)}\frac{|u(x)-u(y)|}{|x-y|^{\omega}} \leq C [(K R_0^{sp})^{\frac{1}{p-1}} + Q(u;x_0,2R_0)] R_0^{-\omega}.$$}
\end{pro}
\begin{pro}(Theorem 1.4, \cite{BLS})\label{pro:large-p}
Let $p\in [2,\infty)$ and $u \in W^{s,p}_{loc}(\Omega) \cap L^\infty_{loc}(\Omega) \cap L^{p-1}(\mathbb{R}^N)$ be a local weak solution of $(-\Delta)^s_p u = f$ in $\Omega$ with $f \in L^\infty_{loc}(\Omega).$ Then $u \in C^{\omega}_{loc}(\Omega)$ for every $0< \omega < \min\{\frac{sp}{p-1},1\}.$ More precisely, for every $0< \omega < \min\{\frac{sp}{p-1},1\}$ and every ball $B_{4R}(x_0) \Subset \Omega$, there exists a constant $C= C(N,s,p,\omega)$ such that
$$[u]_{C^{\omega}(B_{\frac{R}{8}}(x_0))} \leq C [(\|f\|_{L^\infty(B_R(x_0))} R^{sp})^{\frac{1}{p-1}} + Q(u;x_0,R)] R^{-\omega}.$$
\end{pro}
\noindent Moreover we recall the following result which is suitable  for the acquisition  estimates of Theorem \ref{thpr} and  Theorem \ref{thq}.
\begin{Lem}(Lemma 2.5, \cite{IMS})\label{equi:stro-weak}
Let $u \in \overline{W}^{s,p}_{loc}(\Omega)$. For $\epsilon>0$, let $A_\epsilon \subset \mathbb{R}^N \times \mathbb{R}^N$ be a neighbourhood of $D$, the diagonal of $\mathbb{R}^N \times \mathbb{R}^N$, which satisfies
\begin{enumerate}
    \item[(i)] $(x,y) \in A_\epsilon$ for all $(y,x) \in A_\epsilon$,
    \item[(ii)] $\max\left\{\sup_{x \in A_\epsilon} \dist(x,D), \sup_{y \in D} \dist(y,A_\epsilon)\right\} \to 0$ as $ \epsilon \to 0^+.$
\end{enumerate}
For all $x \in \mathbb{R}^N$, we define $A_\epsilon(x)= \{y \in \mathbb{R}^N: (x,y) \in A_\epsilon\}$ and 
$$f_\epsilon(x) = 2\int_{(A_\epsilon(x))^c} \frac{[u(x)-u(y)]^{p-1}}{|x-y|^{N+sp}} ~dy.$$
Assume that $f_\epsilon \to f$ in $L^1_{loc}(\Omega)$. Then, $u$ satisifies
$$(-\Delta)^s_p\, u =  f \ \ \  \text{E-weakly in}\ \Omega.$$
\end{Lem}

\end{document}